\documentclass[a4paper,11pt,reqno]{amsart}

\usepackage{
a4wide,
amsfonts,
amssymb,
amscd,
amsmath,
latexsym,
amsbsy,
enumitem
}

\newtheorem{thm}{Theorem}[section]

\newtheorem{lem}[thm]{Lemma}
\newtheorem{cor}[thm]{Corollary}
\newtheorem{prop}[thm]{Proposition}

\theoremstyle{definition}
\newtheorem{rmk}[thm]{Remark}

\numberwithin{equation}{section}

\def\al{\alpha}
\def\be{\beta}
\def\ga{\gamma}
\def\de{\delta}
\def\la{\lambda}
\def\La{\Lambda}
\def\Ga{\Gamma}
\def\De{\Delta}
\def\R{\mathbb{R}}
\def\C{\mathbb{C}}
\def\N{\mathbb{N}}
\def\cA{\mathcal A}
\def\cD{\mathcal D}
\def\cV{\mathcal V}
\def\diag{\text{\rm{diag}}}
\newcommand{\rFs}[5]{\,_{#1}F_{#2} \left( \genfrac{.}{.}{0pt}{}{#3}{#4};#5 \right)}
\newcommand{\SU}{\mathrm{SU}}

\title[Matrix valued Hermite polynomials]{Matrix valued Hermite polynomials, Burchnall formulas and
non-abelian Toda lattice}

 \author{Mourad E.H. Ismail}
 \address{University of Central Florida, Orlando, Florida 32816, USA}
\email{ismail@math.ucf.edu}
 \author{Erik Koelink}
 \address{IMAPP, Radboud Universiteit, PO Box 9010,
6500 GL Nijmegen,
 the Netherlands}
 \email{e.koelink@math.ru.nl}
 \author{Pablo Rom\'an}
 \address{CIEM,
 FaMAF, Universidad Nacional de C\'ordoba, Medina Allende s/n Ciudad
 Universitaria, C\'ordoba, Argentina}
 \email{roman@famaf.unc.edu.ar}
\date{\today}

\begin{document}

\begin{abstract} A general family of matrix valued Hermite type orthogonal
polynomials is introduced and studied in detail by deriving Pearson equations
for the weight and matrix valued differential equations for these matrix polynomials.
This is used to derive
Rodrigues formulas, explicit formulas for the squared norm and to
give an explicit expression of the matrix entries as well
to derive a connection formula for the matrix polynomials of Hermite type.
We derive matrix valued
analogues of Burchnall formulas in operational form as well explicit expansions
for the matrix valued Hermite type orthogonal
polynomials as well as for previously introduced matrix valued Gegenbauer type orthogonal
polynomials. The Burchnall approach gives two descriptions of the matrix valued
orthogonal polynomials for the Toda modification of the matrix weight for the
Hermite setting. In particular, we obtain a non-trivial solution to the
non-abelian Toda lattice equations.
\end{abstract}

\maketitle


\section{Introduction}\label{sec:intro}

Matrix-valued orthogonal polynomials have been introduced by M.G.~Krein at the end of the 1940's,
see \cite{Krein1}, \cite{Krein2}, in the study of the matrix moment problem and
in the study of self-adjoint operators with higher deficiency indices.
There have been various papers introducing and studying different aspects of matrix valued
orthogonal polynomials, and an extensive overview paper with many references is \cite{DamaPS}.
Matrix valued orthogonal polynomials have since their introduction seen many other
connections and applications, such as e.g. in relation to higher order recurrence relations
\cite{DuraVA}, spectral analysis \cite{ApteN}, \cite{GroeIK}, \cite{GroeK}, scattering theory \cite{ApteN}, \cite{Gero},
representations of compact symmetric spaces e.g. \cite{GrunPT},
\cite{HeckvP}, \cite{KoelvPR1}, \cite{KoelvPR2},
\cite{Prui},
and see \cite{DamaPS} for more applications and references.
The basics of matrix valued orthogonal polynomials is recalled in Section \ref{ssec:MVOPS}.

In a previous paper \cite{IKR}, we have studied the generalization of Burchnall's formulas
for the Hermite polynomials \cite{Burc} to polynomial families in the Askey-scheme and
its $q$-analogue, see e.g. \cite{AskeW}, \cite{KoekLS}, \cite{KoekS} for the Askey scheme.
In particular, one finds some standard solutions to the Toda lattice equations \cite[\S 7]{IKR}, using the
Lax pair formalism and the relation to orthogonal polynomials where the orthogonality measure
is deformed by $e^{-xt}$, see e.g. \cite{BabeBT}, \cite[\S 2.8]{Isma}.
One of the goals is to extend this relation to the non-abelian Toda lattice, which is studied by
e.g. Bruschi et al. \cite{BrusMRL},
Gekhtman \cite{Gekh}. Bruschi et al. \cite{BrusMRL} and Gekhtman \cite{Gekh} prove the integrability, and in Section \ref{ssec:NAToda}
we show that for matrix valued orthogonal polynomials one can obtain solutions of the
non-abelian Toda lattice by considering the three-term recurrence relation for monic matrix valued
orthogonal polynomials where the matrix weight has been deformed by a similar exponential.
This result of Proposition \ref{prop:MV-Toda} might be well-known, and we include a proof for
completeness.
Several other non-abelian Toda lattices are available in the literature, see e.g.
\cite{AlvaAGAMM} and references given there.

The important aspect of the polynomials in the Askey-scheme that make the extension of Burchnall's approach possible is the Rodrigues formula, which is closely related to shift operators, i.e. raising and lowering
operators. For the matrix analogues for Gegenbauer polynomials studied in \cite{Koe:Rio:Rom} such a
Rodrigues formula exists, and the important ingredient to make this work are the Pearson equations.
To prove Pearson equations, the $LDU$-decomposition of the weight in \cite{Koe:Rio:Rom} is essential.
The lowering operator in these cases is the derivative, so that we are in the situation of
Cantero, Moral and Vel\'azquez \cite{CantMV}. The case of the matrix valued Hermite type polynomials
of Section \ref{sec:MV_Hermite} does
not fit into any of their examples.
In \cite{PruijR} another approach of generalizing the results of \cite{KoelvPR1}, \cite{KoelvPR2},
is discussed. In particular, it is shown that the same polynomials of \cite{Koe:Rio:Rom} are obtained
in this way, and in \cite{PruijR} another decomposition of the weight is used. However,
shift operators and Rodrigues formulas are obtained.

In order to obtain an additional example we introduce a general matrix valued analogue of
the Hermite polynomials in Section \ref{sec:MV_Hermite}. We introduce additional degrees of
freedom in both the $L=U^\ast$ and the diagonal part of the $LDU$-decomposition of the matrix weight.
The choice of the $L$-part is motivated as a limit transition of the $L$-part of the weight in
\cite{Koe:Rio:Rom}, where the entries are given in terms of scalar-valued Gegenbauer polynomials.
We then impose natural, but non-linear, conditions on the degrees of freedom in order to
make the Pearson equations work. Then we can find shift operators, explicit expression for
the squared norm, the three-term recurrence relation, etc. An important ingredient are matrix valued
differential operators having the matrix valued Hermite type polynomials as eigenfunctions.
In particular, this leads to an explicit expression for the matrix entries of
matrix valued Hermite type orthogonal polynomials in terms of a sum of a product of two
Hermite polynomials times a dual Hahn polynomial, see Theorem \ref{thm:MVinHermite}.
We prove that there are at least three families of solutions to these conditions, giving rise to
three different families of matrix valued Hermite type polynomials. For one of these families
a special case is closely related to Dur\'an \cite{DuraG}, \cite{Dura-CA}.

In Section \ref{sec:Burchnall_formula} we then extend the general identities of
\cite[\S 2]{IKR} to the matrix valued case, and we discuss briefly the results for the
matrix valued Gegenbauer type orthogonal polynomials of \cite{Koe:Rio:Rom} and
for the newly introduced matrix valued Hermite type orthogonal polynomials.
There are many other cases known of Rodrigues type equations,
but often the formulas involve
additional ingredients which make the approach of this paper unsuitable for these polynomials.
Then we give a non-trivial solution for the non-abelian Toda lattice equations in
Section \ref{sec:explsolNAToda} by looking at the Toda modification of the weight
for the matrix valued Hermite type polynomials. Since there are two expressions
for the matrix valued orthogonal polynomials for the Toda modification, we obtain
non-trivial identities for these polynomials.

\medskip\noindent
\textbf{Acknowledgements.}
Mourad Ismail acknowledges the  research
support and hospitality of Radboud University during his visits which initiated this collaboration.
Erik Koelink gratefully acknowledges the support and hospitality of FaMAF at Universidad Nacional
de C\'ordoba, the Visiting Professor Program 2017--FaMAF and the support of an Erasmus+ travel grant.
The work of Pablo Rom\'an was supported by Radboud Excellence Fellowship, CONICET grant PIP 112-200801-01533, FONCyT grant PICT 2014-3452 and by SeCyT-UNC.
We thank Bruno Eijsvoogel for useful discussions.


\section{Preliminaries}
\label{sec:preliminaries}

In Section \ref{ssec:MVOPS} we recall some basics on matrix valued orthogonal polynomials, especially
we discuss the corresponding three-term recurrence relation and matrix valued differential operators.
General references for Section \ref{ssec:MVOPS} are e.g. \cite{ApteN}, \cite{DamaPS}, \cite{DuraG},
\cite{GrunT}.
In Section \ref{ssec:NAToda} we recall the non-abelian Toda lattice, see Gekhtman \cite{Gekh}
who attributes it to Polyakov, and we show that the standard
Toda modification of the orthogonality measure also works in the matrix case, see \cite[\S 2.8]{Isma}
for the scalar case. Gekhtman \cite{Gekh} proves the integrability in the sense of e.g. \cite{BabeBT},
see also \cite{BrusMRL}.


\subsection{Matrix valued orthogonal polynomials}\label{ssec:MVOPS}

For any $N\in \N\setminus\{0\}$, let $W$ be a complex $N\times N$ matrix valued integrable function on the interval $(a,b)$, where $a$ and $b$ are allowed to be infinite, such that
$W(x)>0$ is positive definite almost everywhere and with finite moments of all orders. In this paper, all matrix weight functions are at least $C^2(a,b)$, and we assume in this
section that $W$ is continuous as well.
We denote by $\text{\rm{Mat}}_N(\C)$ the algebra of all $N\times N$ complex matrices and by $\text{\rm{Mat}}_N(\C)[x]$ the algebra over $\C$  of all polynomials
in $x$ with coefficients in $\text{\rm{Mat}}_N(\C)$.

We consider the Hermitian sesquilinear form $\langle\cdot,\cdot\rangle$ in the linear space $\text{\rm{Mat}}_N(\C)[x]$ given by:
\begin{equation}
\label{eq:HermitianForm}
\langle P,Q \rangle=\int_a^b P(x)W(x)Q(x)^\ast dx.
\end{equation}
In particular, we assume that all moments are finite, and hence that all polynomials are integrable.
For all $P,Q,R\in \text{\rm{Mat}}_N(\C)[x]$, $T\in \text{\rm{Mat}}_N(\C)$ and $a,b\in \C$, the following properties are satisfied:
\begin{itemize}\setlength{\itemsep}{1.5mm}
\item $\langle aP+bQ,R\rangle=a\langle P,R\rangle+b\langle Q,R\rangle$,
\item $\langle TP,Q\rangle=T\langle P,Q\rangle$,
\item $\langle P,Q\rangle^\ast =\langle Q,P\rangle$,
\item $\langle P,P\rangle\geq0$ for all $P\in \text{\rm{Mat}}_N(\C)[x]$. Moreover,
if $\langle P, P\rangle=0$ then $P=0$,
\end{itemize}
see for instance \cite{GrunT}. Given a weight matrix $W$ there exists a unique sequence of monic orthogonal
polynomials $\{P_n\}_{n\geq 0}$ in $\text{\rm{Mat}}_N(\C)[x]$. Moreover, any other
sequence of $\{R_n\}_{n\geq 0}$ of orthogonal polynomials in $\text{\rm{Mat}}_N(\C)[x]$ is
of the form $R_n(x)=E_nP_n(x)$ for some $E_n\in \operatorname{GL}_N(\C)$.

The monic orthogonal polynomials $\{P_n\}_{n\geq 0}$ satisfy the  orthogonality relations
\begin{equation}
\langle P_n,P_m \rangle=\int_a^b P_n(x)W(x)P_m(x)^\ast dx =\de_{m,n} H_n, \qquad n,m\in\N,
\end{equation}
where $H_n>0$. The monic orthogonal polynomials
satisfy a three-term recurrence relation
$$x
P_n(x)=P_{n+1}(x)+B_{n}(x)P_n(x)+C_nP_{n-1}(x), \quad n\geq 0,
$$
where $P_{-1}=0$ and $B_n$, $C_n$ are matrices depending on $n$ and not on $x$.
One property is that $B_nH_n$ is self-adjoint, or $B_nH_n = H_nB_n^\ast$, since
$H_n$ is self-adjoint.

The explicit matrix valued orthogonal polynomials of Section \ref{sec:MV_Hermite} are
studied using matrix valued differential operators.
Let $D$ be a matrix valued second-order differential operator which acts from the right of the following form: there exist matrix functions $F_i$ in
$C^2([a,b])$ such that for all matrix valued $C^{2}([a,b])$-function $Q$ we have
\begin{equation}
\label{eq:form-differential-operator-general}
QD=\left(\frac{d^2Q}{dx^2}\right)(x) \, F_2(x) + \left(\frac{dQ}{dx}\right)(x) \, F_1(x) +Q(x) F_0(x),
\end{equation}
where derivatives of matrix valued functions are taken entry-wise.
We say that $D$ is symmetric with respect to $W$ if for all matrix valued $C^{2}([a,b])$-functions $G,H$ we have
$$
\int_a^b (GD)(x)W(x)(H(x))^\ast\, dx = \int_a^b G(x)W(x)((HD)(x))^\ast, dx.
$$
By \cite[Thm~3.1]{DuraG},  $D$ is symmetric with respect to $W$ if and only if
the boundary conditions
\begin{gather}
\label{eq:symmetry-boundary1}
\lim_{x\to a} F_2(x)W(x) = 0 = \lim_{x\to b} F_2(x)W(x), \\
\label{eq:symmetry-boundary2}
\lim_{x\to a} F_1(x)W(x) - \frac{d(F_2W)}{dx}(x) = 0 =
\lim_{x\to b} F_1(x)W(x) - \frac{d(F_2W)}{dx}(x)
\end{gather}
and the symmetry conditions
\begin{gather}
\label{eq:symmetry-conditions}
F_2(x)W(x) = W(x) \bigl( F_2(x)\bigr)^\ast, \qquad
2 \frac{d(F_2W)}{dx}(x) - F_1(x)W(x) = W(x) \bigl( F_1(x)\bigr)^\ast, \\
\label{eq:symmetry-conditions2}
\frac{d^2(F_2W)}{dx^2}(x) - \frac{d(F_1W)}{dx}(x) + F_0(x) W(x) = W(x) \bigl( F_0(x)\bigr)^\ast
\end{gather}
for all $x\in (a,b)$  hold.

\begin{rmk}
\label{rmk:conjugation-of-D}
Assume that we have a weight matrix of the form $W(x)=L(x)T(x)L(x)^\ast$ for a certain matrix
functions $L$ and $T$.
Let $\widetilde D =\frac{d^2}{dx^2} \, \widetilde F_2 +\frac{d}{dx}\, \widetilde F_1(x)+\widetilde F_0$ be the differential operator obtained by conjugation of $D$ by $L(x)$. Then
\begin{equation}\label{eq:FLtildeF}
F_2L=L\tilde{F}_2, \qquad
F_1L= 2\frac{dL}{dx} \tilde{F}_2 + L\tilde{F}_1, \qquad
F_0L = \frac{d^2L}{dx^2} \tilde{F}_2 + \frac{dL}{dx} \tilde{F}_1 + L\tilde{F}_0
\end{equation}
By \cite[Prop.~4.2]{Koe:Rio:Rom}, we have that $D$ is symmetric with respect to $W$ if and only if $\widetilde D$ is symmetric with respect to $T$.
\end{rmk}

A matrix weight $W$ is said to be reducible to weights of smaller size if there exists a constant matrix $M$ and weights $W_1,\ldots,W_k$, $k>1$, of size less than  $N$ such that
$MW(x)M^\ast$ is equal to the block diagonal matrix $\mathrm{diag}(W_1(x),\ldots, W_k(x))$ for all $x\in [a,b]$. In such a case, the real vector space
$$
{\cA}_{W}= \{ Y \in \mathrm{Mat}_N(\C)  \mid YW(x) = W(x)Y^\ast, \quad \forall\, x\in [a,b]\},
$$
is non-trivial. On the other hand, if the commutant algebra
$$
A_W= \{Y \in\mathrm{Mat}_N(\C) \mid YW(x) = W(x)Y,\quad \forall\, x\in [a,b]\},
$$
is nontrivial, then the weight $W$ is reducible via a unitary matrix $M$. In \cite{KR15} we
have proved that
$\cA_{W}$ is $\ast$-invariant if and only if  $\cA_{W}$ is equal to the subspace of Hermitian elements of $A_{W}$.
Therefore if $\cA_{W}$ is $\ast$-invariant, then the weight $W$ reduces to weights of smaller size if and only if the commutant algebra $A_{W}$ is not trivial.
See also \cite{TiraZ}.


\subsection{Non-abelian Toda lattice}
\label{ssec:NAToda}

In this section, for a given weight matrix $W$ as in Section \ref{ssec:MVOPS},
we study a deformation of the form
$W(x;t)= e^{-xt\La}W(x)$, and we show that the coefficients of the three-term recurrence relation for the deformed weight
$W(x;t)$ provide a solution for the Toda lattice equations, \cite[\S 1]{BrusMRL}, \cite[\S 1]{Gekh}.

For a fixed constant matrix $\La$, we consider the matrix weight $W(x;t)= e^{-xt\La}W(x)$, supported on the
real line with respect to a positive measure $\mu$ supported on a set $S$.
We assume the support $S$ is not finite.
Moreover, we take $t\in I\subset \R$, for $I$ an open interval containing $0$.
We assume that all moments the matrix valued measure $W(x;t)\, d\mu(x)$ exist for all $t\in I$.
Since we want to study orthogonal polynomials with respect to $W(x;t)$ we assume $W(x;t)$ is a positive definite matrix.
This in turn implies that $\La$ satisfies
$$
\La W(x;t) = W(x;t) \La^\ast\quad \Longrightarrow \quad \La W(x) = W(x) \La^\ast.
$$
Indeed, differentiating $W(x,t)^\ast = W(x,t)$ with respect to $t$ and using that $W(x)^\ast=W(x)$ gives the result.
The final observation follows by setting $t=0$.
In other words, $\La$ belongs to the real vector space
$$
{\cA}_t= \{ Y \in \text{\rm{Mat}}_N(\C) \mid YW(x;t) = W(x;t)Y^\ast, \,  \forall x\in S\}
$$
for all $t\in I$. In particular, $\La \in {\cA}_0={\cA}$.
Note that $\La$ normal and element of $\cA$ imply that $\La\in \cA_t$ and conversely, $\La\in \cA_t$ for 
all $t$ yield that  $\La$ is normal and $\La\in\cA$, because of the specific form of $W(x;t)$.

We let $\{P_n(\cdot;t)\}_n$ be the monic orthogonal polynomials with respect to $W(x;t)\, d\mu(x)$.
Then $P_n(x;t)$ satisfy the following orthogonality relations
\begin{equation}
\label{eq:orthogonality_Pn(xt)}
\int_\R\, P_{n}(x;t) \, W(x)e^{-xt\La} \, P_m(x;t)^\ast \, d\mu(x) = \de_{n,m}\, H_n(t),
\end{equation}
where $H_n(t)$ is a positive matrix for all $t\in I$.
Moreover, the matrix valued orthogonal polynomials satisfy the three-term recurrence relations
\begin{equation}
\label{eq:rec_rel_Toda}
xP_{n}(x;t)= P_{n+1}(x;t)+B_n(t)P_{n}(x;t)+C_n(t)P_{n-1}(x;t).
\end{equation}
By \cite[Lemma 3.1]{KR15} the following relations hold true:
\begin{equation}
\label{eq:commutationBnHn}
\La B_n(t)=B_n(t)\La,\quad \La C_n(t)=C_n(t)\La, \quad \La H_n(t)=H_n(t)\La^\ast,
\quad \La P_n(x;t)=P_n(x;t)\La
\end{equation}
since $\La \in {\cA}_t$.

Proposition \ref{prop:MV-Toda} shows that matrix valued functions $B_n$ and $C_n$ satisfy essentially
the coupled differential equations of the non-abelian Toda lattice in Gekhtman \cite{Gekh}, who attributes
the equations to Polyakov. We derive the non-abelian Toda lattice equations using the
Toda modification of the matrix valued weight measure, see \cite[\S 2.8]{Isma}. In particular, Gekhtman \cite{Gekh} proves the
integrability in the finite case using a Lax pair formalism and we refer to
\cite{Gekh} for more information and references on the non-abelian Toda lattice and
\cite{BabeBT} for general information on integrable systems.

\begin{prop}\label{prop:MV-Toda}
The recurrence coefficients $B_n(t)$ and $C_n(t)$ in \eqref{eq:rec_rel_Toda} satisfy the following differential equations
\begin{align*}
\dot{C}_n(t) &= \La \bigl( C_n(t)B_{n-1}(t) -  B_n(t) C_n(t)\bigr),\\
\dot{B}_n(t) &= \La \bigl( C_n(t) - C_{n+1}(t) \bigr).
\end{align*}
where we use the dot to denote the derivative with respect to $t$.
\end{prop}

\begin{rmk}
Note that $c_n(t) = \La^2 C_n(t)$, $b_n(t) = \La B_n(t)$ satisfy the differential equations
of Proposition \ref{prop:MV-Toda} for $\La=I$, i.e.
\begin{gather*}
\dot{c}_n(t) = c_n(t)b_{n-1}(t) -  b_n(t) c_n(t),\qquad
\dot{b}_n(t) =  c_n(t) - c_{n+1}(t),
\end{gather*}
using \eqref{eq:commutationBnHn}.
\end{rmk}

\begin{proof}
If we put $n=m$ in \eqref{eq:orthogonality_Pn(xt)} and we differentiate with respect to $t$, we obtain
\begin{equation}
\label{eq:HnBn}
\begin{split}
\dot{H}_n(t) &= -\int_\R \, xP_{n}(x;t) \La e^{-xt\La} W(x) P_{n}(x;t)^\ast\, d\mu(x) = - \La B_n(t)H_n(t), \\
& \Longrightarrow \quad
\dot{H}_n(t) = - \La H_n(t)B_n(t)^\ast
\end{split}
\end{equation}
using  \eqref{eq:rec_rel_Toda}, \eqref{eq:commutationBnHn}, the self-adjointness of $H_n(t)$ and $\dot{P}_n(x;t)$ being a
polynomial of degree less than or equal to $n-1$ and \eqref{eq:orthogonality_Pn(xt)}.
On the other hand, we have in general
\begin{align}
C_n(t)H_{n-1}(t)&=\int_\R (xP_n(x;t)-P_{n+1}(x;t)-B_n(t)P_n(x;t)) e^{-xt\La} W(x) P_{n-1}(x;t)^\ast \, d\mu(x), \nonumber\\
&=\int_\R xP_n(x;t) e^{-xt\La} W(x) P_{n-1}(x;t)^\ast \, d\mu(x) \nonumber\\
&= \int_\R P_n(x;t) e^{-xt\La} W(x) (xP_{n-1}(x;t))^\ast \, d\mu(x) =H_n(t).
\label{eq:Hn_CnHnm1}
\end{align}
Differentiating $\int_\R P_n(x;t) e^{-xt\La} W(x) P_{n-1}(x;t)^\ast \, d\mu(x)=0$
with respect to $t$, we get
\begin{equation}
\label{eq:Hn_as_int_Pnm1}
\La H_n(t)=\int_\R \dot{P}_n(x;t) e^{-xt\La} W(x) P_{n-1}(x;t)^\ast \, d\mu(x)
\end{equation}
using \eqref{eq:rec_rel_Toda}, \eqref{eq:orthogonality_Pn(xt)}, \eqref{eq:commutationBnHn}.

Now observe that
$$
B_n(t)H_n(t)=\int_\R xP_n(x;t) e^{-xt\La} W(x) P_{n}(x;t)^\ast \, d\mu(x)
$$
and differentiating both sides with respect to $t$, we obtain
\begin{align*}
\dot{B}_n(t)H_n(t)+B_n(t)\dot{H}_n(t)&=-\int_\R (xP_n(x;t)) \La  e^{-xt\La}W(x) (xP_{n}(x;t))^\ast \, d\mu(x) \\
&\qquad+\int_\R \dot{P}_n(x;t) e^{-xt\La} W(x)(xP_{n}(x;t))^\ast \, d\mu(x) \\
& \qquad \qquad +\int_\R xP_n(x;t) e^{-xt\La} W(x) \dot{P}_{n}(x;t)^\ast \, d\mu(x). 
\end{align*}
The first integral on the right hand side equals
\[
-\La H_{n+1}(t) - \La B_n(t) H_n(t) B_n(t)^\ast - \La C_n(t) H_{n-1}(t) C_n(t)^\ast
\]
using \eqref{eq:orthogonality_Pn(xt)}, \eqref{eq:rec_rel_Toda}
and \eqref{eq:commutationBnHn}.
The remaining two integrals are each others adjoints, and by \eqref{eq:rec_rel_Toda}, \eqref{eq:Hn_as_int_Pnm1}
and \eqref{eq:orthogonality_Pn(xt)} they equal $\La H_n(t)C_n(t)^\ast + C_n(t)H_n(t) \La^\ast$.
Collecting the terms and reordering, using \eqref{eq:HnBn}, \eqref{eq:Hn_CnHnm1} twice
and \eqref{eq:commutationBnHn}, we find
\begin{align*}
\dot{B}_n(t)H_n(t)&=\La B_n(t) H_n(t) B_n(t)^\ast -\La H_{n+1}(t)-\La B_n(t)H_n(t) B_{n}(t)^\ast -\La C_n(t)H_{n-1}(t) C_n(t)^\ast\\
& \qquad + \La H_n(t)C_n(t)^\ast + \La C_n(t)H_n(t) \\
& = -\La H_{n+1}(t)-\La H_n(t) C_{n}(t)^\ast
+ \La H_n(t)C_n(t)^\ast + \La C_n(t)H_n(t) \\
& = -\La C_{n+1}(t) H_{n}(t)  + \La C_n(t)H_n(t)
\end{align*}
proving the second relation, since $H_n(t)$ is invertible by the 4th property of the matrix valued
orthogonal polynomials in Section \ref{ssec:MVOPS}.

Finally, taking derivatives with respect to $t$ of \eqref{eq:Hn_CnHnm1}
using \eqref{eq:HnBn} we find, using \eqref{eq:commutationBnHn}, \eqref{eq:Hn_CnHnm1},
and $B_nH_n$ being self-adjoint,
\begin{equation*}
\begin{split}
\dot{C}_n(t) H_{n-1}(t) &= \La C_n(t)B_{n-1}(t)H_{n-1}(t) -  \La B_n(t) H_n(t)  \\
&= \La C_n(t)B_{n-1}(t)H_{n-1}(t) - \La B_n(t) C_n(t) H_{n-1}(t)
\end{split}
\end{equation*}
and using the invertibility of $H_{n-1}(t)$ gives the first relation.
\end{proof}

The proof of Lemma \ref{lem:Todapolstimederivative} follows the classical
proof, see e.g. \cite{Belo}, and we give the proof for completeness.

\begin{lem}\label{lem:Todapolstimederivative}
Let $P_n(x;t) = Ix^n + X_n(t) x^{n-1}+\text{\rm{l.o.t.}}$, then
\[
\dot{P_n}(x;t) =
\frac{\partial{P_n}}{\partial t}(x;t) = \frac{d{X_n}}{dt}(t)\,  P_{n-1}(x;t).
\]
\end{lem}

\begin{proof} Note $\int_\R P_n(x;t) e^{-xt\La} W(x) x^kI\, d\mu(x)=0$ for
$k<n$. Differentiate with respect to $t$ to get
\begin{gather*}
\int_\R \dot{P_n}(x;t) e^{-xt\La} W(x) x^kI \, d\mu(x) =
\int_\R P_n(x;t) \La e^{-xt} W(x) x^{k+1}I\, d\mu(x)  \\
= \int_\R P_n(x;t) e^{-xt} W(x) x^{k+1}I\, d\mu(x)\, \La^\ast
\end{gather*}
since $\La\in \cA_t$. Now the right hand side is $0$ for $k+1<n$,
and since $\dot{P_n}(x;t)$ is a polynomial of degree $n-1$, we
see that $\dot{P_n}(x;t) = E_n P_{n-1}(x;t)$. By comparing leading
coefficients the result follows.
\end{proof}


\section{Matrix valued Hermite type polynomials}
\label{sec:MV_Hermite}

In this section we introduce general matrix valued Hermite type polynomials. The scalar Hermite
polynomials are at the lower end in the Askey scheme \cite{AskeW}, \cite{KoekLS}, \cite{KoekS}, meaning that the
Hermite polynomials are hypergeometric functions which are both eigenfunctions to a second-order operator,
in this case a differential operator, as well as eigenfunctions to a three-term recurrence operator.
The Hermite polynomials are at the lower end, since they have no additional degree of freedom besides the
degree and the argument.
The matrix valued polynomials introduced here have some degree of freedom, which we can use to
impose Pearson type equations, see Propositions \ref{prop:Pearson1}, \ref{prop:Pearson2}, for
the matrix weight for these polynomials. So the additional degrees of freedom arise from the
matrix valuedness.
These Pearson equations allow us to generate these polynomials
using shift operators, and to describe the three-term recursion relation
as well other properties explicitly.
Two commuting matrix valued differential operators to which these matrix valued are eigenfunctions
can be used to explicitly express the matrix entries of the matrix valued Hermite type polynomials
as a sum of a product of two (scalar-valued) Hermite polynomials times a Hahn polynomial.
In the process we have to impose certain conditions on the parameters, and we give three different classes
of solutions. A special case of the first class is
closely related to an example studied by Dur\'an \cite{Dura-CA}.


\subsection{The matrix valued weight}
\label{ssec:weightHermite}
Let $L$ be the lower triangular $N\times N$-matrix valued polynomial defined by
\begin{equation}
\label{eq:matrix-L-Hermite}
L(x)_{m,n}=\begin{cases} \frac{H_{m-n}(x)}{(m-n)!}, & m\geq n,\\
0& n<m,
\end{cases}
\qquad m,n\in \{1,\cdots, N\},
\end{equation}
where $H_n(x)$ are the standard Hermite polynomials, see
\cite{AndrAR}, \cite{Isma}, \cite{KoekLS}, \cite{KoekS}, e.g.
defined by the generating function
\begin{equation}\label{eq:genfunHermite}
\sum_{n=0}^\infty \frac{t^n}{n!} H_{n}(x)=e^{2xt-t^2}.
\end{equation}
Note that $L$ is a Toeplitz matrix, i.e. constant on diagonals.
For a sequence $\{\al_1,\cdots,\al_N\}$ of positive numbers, we
define the diagonal matrix $S^\al=\diag(\al_1,\cdots,\al_N)$ and we
set $L^\al(x)= S^\al L(x) (S^\al)^{-1}$.
Since we can choose the sequence up to a common scalar, we normalize $\al_1=1$.
Observe that $L$ and $L^\al$ are lower triangular polynomial matrix functions and that $L$ and $L^\al$ are unipotent. Hence, the inverses
are unipotent lower triangular polynomial matrix functions as well.

\begin{prop}\label{prop:LisexpandLinv}
The inverse of $L$ is given explicitly by
$$(L(x)^{-1})_{m,n}=\begin{cases}
i^{m-n} \, \frac{H_{m-n}(ix)}{(m-n)!}, & m\geq n,\\
0& n<m.
\end{cases}$$
Moreover,
$\frac{dL}{dx}(x)=A\,L(x) =L(x)\, A$, with $A= 2 \sum_{j=2}^N
E_{j,j-1}$.
\end{prop}

Here and elsewhere $E_{i,j}$ denotes the matrix with all zeros, except a $1$ at the
$(i,j)$-th entry.

As stated in Section \ref{sec:intro}, $L$ has been obtained initially by a limit process
from the corresponding $L$ in terms of Gegenbauer polynomials as in \cite{Koe:Rio:Rom}. Then the
inverse of $L$ is explicitly calculated by Cagliero and Koornwinder \cite{CaglK} in terms
of Gegenbauer polynomials of negative parameter. In this case the inverse is given by
Hermite polynomials on the imaginary axis, and so we can consider these as
limits of Gegenbauer polynomials of negative parameter.

\begin{cor}\label{cor:prop:LisexpandLinv}
\label{cor:L-exponential}
$L(x)=L(0)e^{xA}$ and $L^\al (x)=L^\al (0)e^{xA^\al}$, with
$A^\al =S^\al A (S^\al)^{-1}$.
\end{cor}

Note in particular that $A$ and $L(0)$ commute.

\begin{proof}[Proof of Proposition \ref{prop:LisexpandLinv}] Multiply the generating function \eqref{eq:genfunHermite}
with the generating function with $x$ and $t$ replaced by $ix$ and $it$, which gives $1$. Comparing powers of $t^p$ on both sides gives
\begin{equation}\label{eq:Hermiteinverse}
\sum_{m=0}^p i^m \frac{H_{m}(ix)}{m!}\frac{H_{p-m}(x)}{(p-m)!} = \de_{p,0},
\end{equation}
which gives the inverse of $L(x)$.

Differentiating the generating function gives $\frac{dH_n}{dx}(x) = 2n H_{n-1}(x)$, and this gives the first order differential equation for $L$.
\end{proof}

We now introduce a positive definite matrix valued weight with finite moments in terms of the matrix valued polynomial $L$.
Define the weight matrix
\begin{equation}\label{eq:weight_factorization}
W^{(\nu)}(x)=L(x)\,T^{(\nu)}(x)\,L(x)^\ast,
\qquad T^{(\nu)}(x)=e^{-x^2} \De^{(\nu)},
\end{equation}
where $\De^{(\nu)} =\diag(\de_1^{(\nu)},\cdots, \de_N^{(\nu)})$.
The initial condition is that $\de_i^{(\nu)}>0$, so that
$W^{(\nu)}$ is positive definite. Moreover, the term
$e^{-x^2}$ in $T^{(\nu)}$ guarantees that all moments exist,
so that we are in the situation of Section \ref{ssec:MVOPS}.
Note that for real $\al_i$'s
\begin{equation}\label{eq:weight_factorization-alpha}
W^{(\al,\nu)}(x)=L^\al(x)\,T^{(\nu)}(x)\,\bigl( L^\al(x)\bigr)^\ast
= S^\al L(x)\, (S^\al)^{-2} \,T^{(\nu)}(x)\, L(x)^\ast  (S^\al)^{-1},
\end{equation}
since $[T^{(\nu)}(x),S^\al]=0$ as $T^{(\nu)}(x)$ is diagonal.
From now on we assume that $\al_i$
is real and non-zero for all $i$, and we normalize $\al_1=1$.

\begin{rmk} The case $N=1$ reduces the weights \eqref{eq:weight_factorization}, \eqref{eq:weight_factorization-alpha}
to the weight of the Hermite polynomials, so that the
matrix valued orthogonal polynomials for the weights in  \eqref{eq:weight_factorization} and
\eqref{eq:weight_factorization-alpha}
are matrix valued analogues of the Hermite polynomials.
The additional degrees of freedom in the finite sequence
$\{\al_1,\cdots,\al_N\}$ of non-zero real numbers and $\nu$
will allow for Pearson type equations for these weights,
see Propositions \ref{prop:Pearson1} and \ref{prop:Pearson2}.
\end{rmk}

It is clear that the weights $W^{(\al,\nu)}$ and its special case $W^{(\nu)}$
satisfy the conditions of Section \ref{ssec:MVOPS}.
We denote that monic matrix valued orthogonal polynomials
for the weight $W^{(\nu)}$ by $P^{(\nu)}_n$ and
for the weight $W^{(\al,\nu)}$ by $P^{(\al,\nu)}_n$.


\subsection{A symmetric second order differential operator}
Now we derive a second-order differential operator which is symmetric with respect to
the matrix weight $W^{(\al, \nu)}$ and which preserves polynomials and its degree.
The idea is to conjugate to the differential equation for the Hermite polynomials in a diagonal form.

We need the diagonal matrix $J$; $J_{k,l}=\de_{k,l}k$. Then
$[J,A^\al]=A^\al$, and
\begin{equation}
\label{eq:conjAHofJ}
e^{-xA^\al} J e^{xA^\al}= xA^\al+J.
\end{equation}
Indeed, the left hand side is a matrix valued polynomial in $x$, since $A^\al$ is nilpotent. Its derivative
equals $e^{-xA^\al} [J,A^\al] e^{xA^\al}=A^\al$, so \eqref{eq:conjAHofJ} follows.

We also need to relate the matrix $J$ to the lower triangular matrix polynomial $L^\al$.

\begin{lem}
\label{lem:Lm1JL}
For all $x$ we have
$$
\bigl(L^\al(x)\bigr)^{-1}JL^\al(x)=J-\frac{1}{2}(A^\al)^2+xA^\al,\qquad L^\al(x)J\bigl(L^\al(x)\bigr)^{-1}=J+\frac{1}{2}\bigl(A^\al\bigr)^2-xA^\al.
$$
\end{lem}

\begin{proof} Because of Proposition \ref{prop:LisexpandLinv}
and Corollary \ref{cor:prop:LisexpandLinv} this follows from
\eqref{eq:conjAHofJ} and
$$
\bigl(L^\al(0)\bigr)^{-1}JL^\al(0)=
J-\frac{1}{2}(A^\al)^2.
$$
Taking the $(m,n)$-entry, this follows from $H_{n}(0)= -2(n-1)H_{n-2}(0)$
which in turn is a consequence of the three-term recurrence for the Hermite polynomials, see e.g. \cite[(1.13.3)]{KoekS}.
\end{proof}

Proposition \ref{prop:symmetry_D} has various versions in the literature, see e.g.
\cite{Dura-CA}, \cite{DuraG}, \cite{GrundlIMF}. We give another proof using
conjugation as in \cite{Koe:Rio:Rom}.

\begin{prop}
\label{prop:symmetry_D}
Let $D= D^{(\al)}$ be the second-order matrix valued differential operator acting from the right given by
$$
D= \left(\frac{d^2}{dx^2}\right) \, F_{2}(x) +\left(\frac{d}{dx}\right) \, F_{1}(x) +F_{0}(x),
$$
with
$$F_{2}(x)=I,\qquad F_{1}(x)=-2x+2A^\al,\qquad F_{0}(x)=-2J.$$
Then $D$ is symmetric with respect to $W^{(\al,\nu)}$.
Moreover,
$$
P^{(\al,\nu)}_{n} D\, =\,\Ga_nP_{n}^{(\al,\nu)},
\qquad \Ga_n=-2n-2J, \qquad n\in \mathbb{N}.
$$
\end{prop}

Note that the matrix valued differential operator of Proposition \ref{prop:symmetry_D} only depends on $\al$, and that the eigenvalue matrices $\Ga_n$ do not depend on $\al$ nor on $\nu$.

\begin{proof} We start by checking that these values for $F_i$ give the
right values for $\tilde{F_i}$ as in Remark \ref{rmk:conjugation-of-D}
using $\frac{d (L^\al)^{-1}}{dx} =
-A^\al L^\al$, $\frac{d^2 (L^\al)^{-1}}{dx^2} =
(A^\al)^2 L^\al$ by Lemma \ref{lem:Lm1JL}. Trivially, $\tilde{F_2}(x)=I$, and
by Proposition \ref{prop:LisexpandLinv} and Corollary \ref{cor:prop:LisexpandLinv}
\[
\tilde{F_1}(x) = -2 A^\al + \bigl(L^\al(x)\bigr)^{-1} (2A^\al -2xI) L^\al(x)
= -2xI
\]
using \eqref{eq:FLtildeF}. Finally, the last equation of \eqref{eq:FLtildeF}
gives, using Lemma \ref{lem:Lm1JL},
\[
\tilde{F_0}(x) = (A^\al)^2 - A^\al(2A^\al -2x) + (L^\al(x))^{-1} (-2J)
L^\al(x) = -2J.
\]
Now $\tilde{D}$, being diagonal, is symmetric with respect to the
diagonal weight $e^{-x^2}\De^{(\nu)}$ by a scalar calculation.
By the discussion above, the operator $D$ is symmetric with respect to
the matrix weight $W^{(\al,\nu)}$.

Since $D$ preserves polynomials as well as the degree of the polynomials, $P_{n}^{(\al,\nu)}D$ are
also orthogonal with respect to the matrix weight $W^{(\al,\nu)}$,
so that $P_{n}^{(\al,\nu)}D = \Ga_n P_{n}^{(\al,\nu)}$ for
some eigenvalue matrix $\Ga_n$, which is obtained by considering the leading coefficient.
\end{proof}

Since the matrix valued operator $D$ of Proposition \ref{prop:symmetry_D}
is conjugated to $\tilde{D}$ as in Remark \ref{rmk:conjugation-of-D}, we
see that, cf. \cite[\S 4]{Koe:Rio:Rom},
\[
\frac{d^2(P^{(\al,\nu)}_nL^\al)}{dx^2}(x) I -2x \frac{d(P^{(\al,\nu)}_nL^\al)}{dx}(x) I
-2 (P^{(\al,\nu)}_nL^\al)(x) J = \Ga_n  (P^{(\al,\nu)}_nL^\al)(x).
\]
Since all matrices involved in the differential operator are diagonal, we get a second-order differential equation for the $(r,s)$-entry;
\[
\frac{d^2(P^{(\al,\nu)}_nL^\al)_{r,s}}{dx^2}(x) -2x \frac{d(P^{(\al,\nu)}_nL^\al)_{r,s}}{dx}(x)
= (-2n +2s -2r) (P^{(\al,\nu)}_nL^\al)_{r,s}
\]
and since $(P^{(\al,\nu)}_nL^\al)_{r,s}$ is polynomial, it follows that
$(P^{(\al,\nu)}_nL^\al)_{r,s}(x) = c^{(\al,\nu)}_{r,s}(n) H_{n+r-s}(x)$
for certain constants $c^{(\al,\nu)}_{r,s}(n)$, which we determine
explicitly in the proof of Theorem \ref{thm:MVinHermite} in terms
of ${}_3F_2$-series.
Note that $n=0$ corresponds to \eqref{eq:matrix-L-Hermite}.


\subsection{The Pearson equations}
In this subsection we impose conditions on the sequence $\al$ and
the diagonal matrices $\De^{(\nu)}$'s which allow for Pearson
type equations. The Pearson equations allow for shift operators
in \S \ref{ssec:shift},
which in turn give explicit expressions for the polynomials in terms
of Rodrigues formulas, the three-term recurrence and the squared norms.

First we prove that the family of weights $W^{(\al, \nu)}$ satisfies a first Pearson-type equation under suitable conditions. For this we assume that there exist coefficients
$\{c^{(\nu)}\}_{ \nu\geq \nu_0}$ and $\{d^{(\nu)}\}_{ \nu\geq \nu_0}$ such that
$\de^{(\nu+1)}_{k}=(k\, d^{(\nu)} +c^{(\nu)}) \, \de^{(\nu)}_{k}$ for all $k=1,\ldots,N$. In other words, we assume that
\begin{equation}\label{eq:condition_delta_Phi}
\Delta^{(\nu+1)}=(d^{(\nu)} J+c^{(\nu)}) \, \Delta^{(\nu)} =
\Delta^{(\nu)}\, (d^{(\nu)} J+c^{(\nu)}).
\end{equation}
Note that the coefficients $d^{(\nu)}$ and $c^{(\nu)}$
are real, and we assume that $d^{(\nu)}$ and $c^{(\nu)}$ are positive.

\begin{prop}\label{prop:Pearson1}
Let $W^{(\al,\nu)}$ be the weight matrix given in \eqref{eq:weight_factorization-alpha} and assume
\eqref{eq:condition_delta_Phi}.
Then
$$\Phi^{(\al,\nu)}(x)=(W^{(\al,\nu)}(x))^{-1}W^{(\al,\nu+1)}(x),$$
is a matrix valued polynomial of degree one.
\end{prop}

\begin{proof}
It follows from \eqref{eq:weight_factorization-alpha} and \eqref{eq:condition_delta_Phi} that
\begin{gather*}
\Phi^{(\al,\nu)}(x) = (W^{(\al,\nu)}(x))^{-1} W^{(\al,\nu+1)}(x) =
\bigl( L^\al(x)^\ast\bigr)^{-1} (d^{(\nu)} J+c^{(\nu)}) L^\al(x)^\ast \\
= d^{(\nu)} \Bigl( J+ \frac12\bigl((A^\al)^\ast\bigr)^2-x\bigl(A^\al)^\ast
\Bigr) + c^{(\nu)}
\end{gather*}
by taking the adjoint of Lemma \ref{lem:Lm1JL}.
\end{proof}

Note that we can reformulate the result of Proposition \ref{prop:Pearson1} using
\eqref{eq:weight_factorization-alpha} and \eqref{eq:condition_delta_Phi} as
\begin{equation}\label{eq:LinvPhistarL}
\bigl( L^\al(x)\bigr)^{-1} \bigl( \Phi^{(\al,\nu)}(x)\bigr)^\ast
L^\al(x)
= d^{(\nu)} J+c^{(\nu)}.
\end{equation}

Assume next that the coefficients $\al_k$ and $\de^{(\nu)}_k$ satisfy the  relation
\begin{equation}
\label{eq:recursion-alphas}
\frac{\al_{k+1}^2}{\al_{k}^2}=\frac{d^{(\nu)}k(N-k)}{2(d^{(\nu)}k+c^{(\nu)})} \frac{ \de_{k+1}^{(\nu)} }{\de_{k}^{(\nu)}}, \qquad k=1,\ldots,N-1.
\end{equation}

Note that, since the coefficients $\al_k$ are independent of $\nu$, we in particular require that the right hand side of \eqref{eq:recursion-alphas} is independent of $\nu$.

\begin{prop}\label{prop:Pearson2}
Let $W^{(\al,\nu)}$ be the weight matrix given in \eqref{eq:weight_factorization-alpha} and assume
\eqref{eq:condition_delta_Phi} and
\eqref{eq:recursion-alphas}. Then
$$
\Psi^{(\al,\nu)}(x)=(W^{(\al,\nu)}(x))^{-1}\frac{dW^{(\al,\nu+1)}}{dx}(x),
$$
is a matrix valued polynomial of degree one.
\end{prop}

\begin{proof}
By construction $\Psi^{(\al, \nu)}$ is a matrix polynomial.
Using \eqref{eq:weight_factorization-alpha} the derivative gives three terms, and combining the terms shows that
\begin{multline*}
\Psi^{(\al,\nu)}(x) = (W^{(\al,\nu)}(x))^{-1}
\frac{dW^{(\al,\nu+1)}}{dx}(x) = \\
\bigl( L^\al(x)^\ast\bigr)^{-1} \bigl(\De^{(\nu)}\bigr)^{-1}
A^\al \De^{(\nu+1)} L^\al(x)^\ast \\
+ \bigl( L^\al(x)^\ast\bigr)^{-1} \Bigl(
-2x \bigl(\De^{(\nu)}\bigr)^{-1}
\De^{(\nu+1)} + \bigl(\De^{(\nu)}\bigr)^{-1}
\De^{(\nu+1)} (A^\al)^\ast \Bigr) L^\al(x)^\ast
\end{multline*}
Now the last term is dealt with as in the
proof of Proposition \ref{prop:Pearson1} using  \eqref{eq:condition_delta_Phi} and Lemma \ref{lem:Lm1JL}, since
$[(A^\al)^\ast, L^\al(x)^\ast]=0$.
Explicitly,
\begin{multline}
\label{eq:third-term}
\bigl( L^\al(x)^\ast\bigr)^{-1} \Bigl(
-2x \bigl(\De^{(\nu)}\bigr)^{-1}
\De^{(\nu+1)} + \bigl(\De^{(\nu)}\bigr)^{-1}
\De^{(\nu+1)} (A^\al)^\ast \Bigr) L^\al(x)^\ast = \\
\bigl( d^{(\nu)}\big\{J + \frac12 \bigl((A^\al)^\ast\bigr)^2 -x (A^\al)^\ast\bigr\} + c^{(\nu)} \bigr) \bigl( -2x + (A^\al)^\ast\bigr).
\end{multline}
Note that \eqref{eq:third-term} is a matrix polynomial of degree $2$, with leading term
$-2x^2(A^\al)^\ast$. So we need to
show that the first term is of degree $2$ as well, and that the second order term cancels the second order term in \eqref{eq:third-term}.

First observe that the matrix $[(\Delta^{(\nu)})^{-1} A^\al\Delta^{(\nu+1)}, (A^\al)^\ast]$ is a diagonal matrix whose $k$-th diagonal entry is given by
$$
[(\Delta^{(\nu)})^{-1} A^\al\Delta^{(\nu+1)}, (A^\al)^\ast]_{k,k}=
\frac{4 \al_k^2 \de_{k-1}^{(\nu+1)}}{\al_{k-1}^2\de^{(\nu)}_k}-\frac{4 \al_{k+1}^2 \de_{k}^{(\nu+1)}}{\al_{k}^2\de^{(\nu)}_{k+1}}.
$$
Here we follow the convention that for $k=0$ the first term on the right hand side is absent, and for $k=N$ the second term on the right hand side is absent.
Using \eqref{eq:recursion-alphas} and next
\eqref{eq:condition_delta_Phi}
we obtain
$$
[(\Delta^{(\nu)})^{-1} A^\al\Delta^{(\nu+1)}, (A^\al)^\ast]
= 4d^{(\nu)} J -2(N+1)d^{(\nu)}.
$$
Using Proposition \ref{prop:LisexpandLinv} and
Corollary \ref{cor:prop:LisexpandLinv} we  have the first equality in
\begin{equation*}
\begin{split}
& \frac{d}{dx} \left( \bigl( L^\al(x)^\ast\bigr)^{-1} \bigl(\De^{(\nu)}\bigr)^{-1} A^\al \De^{(\nu+1)} L^\al(x)^\ast  \right)
=  \bigl( L^\al(x)^\ast\bigr)^{-1}  \bigl(4d^{(\nu)} J -2(N+1)d^{(\nu)}\bigr) L^\al(x)^\ast\\
&\qquad\qquad
= 4d^{(\nu)}\Bigl(J + \frac12 \bigl((A^\al)^\ast\bigr)^2 -x (A^\al)^\ast\Bigr)  -2(N+1)d^{(\nu)},
\end{split}
\end{equation*}
where the second equality follows from Lemma \ref{lem:Lm1JL}.
So we find
\begin{multline}
\label{eq:first-term}
\bigl( L^\al(x)^\ast\bigr)^{-1} \bigl(\De^{(\nu)}\bigr)^{-1} A^\al \De^{(\nu+1)} L^\al(x)^\ast
= \\
2d^{(\nu)}\Bigl(2Jx + \bigl((A^\al)^\ast\bigr)^2x -x^2 (A^\al)^\ast\Bigr)  - 2(N+1)d^{(\nu)}x + \\
\bigl(L^\al(0)^\ast\bigr)^{-1} \bigl(\De^{(\nu)}\bigr)^{-1} A^\al \De^{(\nu+1)} L^\al(0)^\ast
\end{multline}
The sum of \eqref{eq:first-term} and \eqref{eq:third-term} yields
$\Psi^{(\al, \nu)}$ as a polynomial of degree one.
\end{proof}

Note that the proofs of Propositions \ref{prop:Pearson1} and
\ref{prop:Pearson2} give explicit expressions for the
degree $1$ polynomials $\Phi^{(\al, \nu)}$ and
$\Psi^{(\al, \nu)}$. In particular,
\begin{equation}\label{eq:lcPsialnu}
\begin{split}
&\Psi^{(\al, \nu)}(x) =
2x \bigl( d^{(\nu)}(J-(N+1))-c^{(\nu)}\bigr) \\
 + (d^{(\nu)}J+c^{(\nu)})&(A^\al)^\ast + \frac12 d^{(\nu)}
\bigl((A^\al)^3\bigr)^\ast +
\bigl(L^\al(0)^\ast\bigr)^{-1} \bigl(\De^{(\nu)}\bigr)^{-1} A^\al \De^{(\nu+1)} L^\al(0)^\ast.
\end{split}
\end{equation}
Note that the leading coefficient $\text{lc}(\Psi^{(\al, \nu)})$
is a diagonal invertible matrix. Moreover,
$\Psi^{(\al, \nu)}(x)$ is a (upper) Hessenberg matrix, where the
subdiagonal is given by
$\bigl(\De^{(\nu)}\bigr)^{-1} A^\al \De^{(\nu+1)}$ and the diagonal
is the leading term of $\Psi^{(\al, \nu)}$.
In Proposition \ref{prop:3termrecur} we show that
$\Psi^{(\al, \nu)}(0)$ is actually a tridiagonal matrix with
zero diagonal, and this feature is not clear from
\eqref{eq:lcPsialnu}.

For later use we note that the first step gives
\begin{equation}\label{eq:LinvPsistarL}
\bigl( L^\al(x)\bigr)^{-1} \bigl( \Psi^{(\al,\nu)}(x)\bigr)^\ast
L^\al(x)
= (A^\al-2x)(d^{(\nu)}J+c^{(\nu)})
+ \De^{(\nu+1)}(A^\al)^\ast (\De^{(\nu)})^{-1}.
\end{equation}

Apart from the Pearson equation in Proposition \ref{prop:Pearson2}
the condition \eqref{eq:recursion-alphas} also implies
\begin{equation}\label{eq:deoveralsquared}
\frac{\de^{(\nu)}_k}{\al^2_k} = 2^{k-1} \frac{(1+\frac{c^{(\nu)}}{d^{(\nu)}})_{k-1}}{(k-1)! \, (N-k+1)_{k-1}} \de^{(\nu)}_1,
\end{equation}
since $\al_1=1$,
which leads to a proof of Lemma \ref{lem:explWalnumn}, cf. \cite[Def.~2.1]{Koe:Rio:Rom}.

\begin{lem}\label{lem:explWalnumn}
Assuming \eqref{eq:recursion-alphas} we have
\[
\bigl( W^{(\al,\nu)}(x)\bigr)_{m,n}
= \al_m \al_n e^{-x^2} \de^{(\nu)}_1
\sum_{t=1}^{m\wedge n}
\frac{ 2^{t-1} (-N-\frac{c^{(\nu)}}{d^{(\nu)}})_{t-1} H_{m+n-2t}(x)}
{(m-t)! (n-t)! (t-1)! (-N+1)_{t-1}}
\]
\end{lem}

Lemma 3.8 is the analogue of \cite[Def.~2.1, Thm.~2.2]{Koe:Rio:Rom},
\cite[Thm.~2.1]{KoelvPR2} for the matrix valued Gegenbauer type polynomials.
The underlying integral evaluation of a product of three Gegenbauer polynomials is interpreted in terms of a dual addition formula by
Koornwinder \cite{Koor-dualAF}. In \cite{Koor-dualAF} the limit case to
Hermite polynomials is considered, which is the linearization for
Hermite polynomials. This is precisely the formula we need to prove
Lemma \ref{lem:explWalnumn}.

\begin{cor}\label{cor:lem:explWalnumn}
 The $0$-th moment is
\begin{equation*}
\left( H^{(\al,\nu)}_0\right)_{m,n} =
\left( \int_\R W^{(\al,\nu)}(x)\, dx \right)_{m,n} =
\de_{m,n} \al^2_m \de^{(\nu)}_1 2^{m-1}\sqrt{\pi}
\frac{(-N-\frac{c^{(\nu)}}{d^{(\nu)}})_{m-1} }{(m-1)!\, (1-N)_{m-1}}.
\end{equation*}
\end{cor}

This follows since the integral of the expression in
Lemma \ref{lem:explWalnumn} is non-zero if and only $m+n-2t=0$.
In particular, the positivity condition on $c^{(\nu)}$ and $d^{(\nu)}$
shows that the $0$-th moment $H^{(\al,\nu)}_0$ is invertible.

\begin{proof}[Proof of Lemma \ref{lem:explWalnumn}] By definition \eqref{eq:weight_factorization-alpha} we
have
\begin{equation*}
\begin{split}
\bigl(W^{(\al,\nu)}(x)\bigr)_{m,n} & =
\sum_{k=1}^{m\wedge n} \frac{\al_m\al_n}{\al_k^2} \de^{(\nu)}_k e^{-x^2}
\frac{H_{m-k}(x)}{(m-k)!} \frac{H_{n-k}(x)}{(n-k)!} \\
& = \sum_{k=1}^{m\wedge n} \sum_{r=0}^{(m-k)\wedge (n-k)}
\frac{\al_m\al_n}{\al_k^2} \de^{(\nu)}_k e^{-x^2}
 \frac{2^r\, H_{m+n-2(k+r)}}{(m-k-r)! (n-k-r)! r!}
\end{split}
\end{equation*}
using the linearization formula $H_k(x)H_l(x) = \sum_{r=0}^{k\wedge l}
\binom{k}{r}\binom{l}{r} 2^r r!H_{k+l-2r}(x)$. Now put $t=k+r$, and we keep
the remaining sum over $k$. The remaining sum over $k$ becomes, using
\eqref{eq:deoveralsquared},
\[
\sum_{k=1}^t \frac{\de^{(\nu)}_k}{\al_k^2} \frac{2^{t-k}}{(t-k)!}
=
\frac{\de^{(\nu)}_1 2^{t-1}}{(t-1)!}
\sum_{k=1}^t
\frac{(1+\frac{c^{(\nu)}}{d^{(\nu)}})_{k-1}\, (-(t-1))_{k-1} }
{(k-1)! \, (N-k+1)_{k-1}}
\]
and the sum
is a terminating hypergeometric ${}_2F_1$ series evaluated at $1$. Since
such a ${}_2F_1$-series can be summed by the Chu-Vandermonde summation, see
e.g. \cite[Cor.~2.2.3]{AndrAR}, \cite[(1.4.3)]{Isma}, we have that the sum
equals
$\frac{(-N-\frac{c^{(\nu)}}{d^{(\nu)}})_{t-1} }
{(-N+1)_{t-1}}$. Collecting the terms gives the result.
\end{proof}


\subsection{Shift operators}\label{ssec:shift}

With the Pearson equations as in Propositions
\ref{prop:Pearson1} and \ref{prop:Pearson2} we
can consider the corresponding shift operators.
By
\[
\langle P, Q \rangle^{(\nu)} =
\int_\R P(x) W^{(\al,\nu)}(x) Q(x)^\ast\, dx
\]
we denote the matrix valued inner product for matrix valued functions $P$ and $Q$, whenever the integral converges. In particular, this is valid for $P$ and
$Q$ polynomials. We suppress the dependence on $\al$, since the shifts only involve $\nu$.

\begin{prop}
\label{prop:shift}
Assume \eqref{eq:condition_delta_Phi} and
\eqref{eq:recursion-alphas} for all $\nu$ of the form $\nu_0+k$, $k\in \N$,
for some fixed $\nu_0$.
\begin{enumerate}
\item[(i)] The first order matrix valued differential operator $S^{(\nu)}$ defined by
$$
(QS^{(\al,\nu)})(x)=\frac{dQ}{dx}(x)(\Phi^{(\al,\nu)}(x))^\ast +Q(x)(\Psi^{(\al,\nu)}(x))^\ast ,
$$
satisfies
$$\langle \frac{dP}{dx}, Q\rangle^{(\nu+1)} =  -\langle P, QS^{(\al,\nu)}\rangle^{(\nu)},$$
for matrix valued polynomials $P$ and $Q$.
\item[(ii)] The following relations hold
\begin{align*}
\frac{dP^{(\al,\nu)}_{n}}{dx}(x)&=n\, P_{n-1}^{(\al,\nu+1)}(x), \qquad (P^{(\al,\nu+1)}_{n-1}S^{(\nu)})(x)=K^{(\al,\nu)}P_{n}^{(\al,\nu)}(x),
\end{align*}
where $K^{(\al,\nu)}= \text{\rm{lc}}(\Psi^{(\al,\nu)})^\ast$ is
a diagonal invertible matrix.
\item[(iii)] The polynomials $P_{n}^{(\nu)}$ satisfy the following Rodrigues formula
$$
P^{(\al,\nu)}_{n}(x)=G^{(\al,\nu)}_{n} \left( \frac{d^nW^{(\al,\nu+n)}}{dx^n}(x)\right) W^{(\al,\nu)}(x)^{-1},
$$
where $G^{(\al,\nu)}_{n}$ is a diagonal invertible matrix;
$
G^{(\al,\nu)}_{n}=(K^{(\al,\nu)})^{-1}\cdots (K^{(\al, \nu+n-1)})^{-1}.
$
\item[(iv)] The squared norm $H_{n}^{(\al,\nu)}$ is given by
the diagonal invertible matrix
$$
H_{n}^{(\al,\nu)}=(-1)^n \, n! \, (K^{(\al,\nu)})^{-1} (K^{(\al,\nu+1)})^{-1} \cdots (K^{(\al,\nu+n-1)})^{-1}  H_{0}^{(\al,\nu+n)}.
$$
\end{enumerate}
\end{prop}

The proof of Proposition \ref{prop:shift} is
completely analogous to \cite[Cor.~2.5 and Thm.~3.1]{Koe:Rio:Rom},
and we leave this to the reader.

As an immediate corollary we find a second-order matrix valued differential operator for the
matrix valued orthogonal polynomials.

\begin{cor}
\label{cor:second_DO_Hermite}
Assume the conditions as in Proposition \ref{prop:shift}.
The differential operator
$$
\mathcal{D}^{(\al,\nu)} =
\left( \frac{d^2}{dx^2} \right) \Phi^{(\al,\nu)}(x)^\ast
+ \left( \frac{d}{dx} \right) \Psi^{(\al,\nu)}(x)^\ast
$$
is symmetric with respect to the weight $W^{(\al,\nu)}$. Moreover,
$$
P_{n}^{(\al,\nu)}\mathcal{D}^{(\al,\nu)}=
\La^{(\al,\nu)}_nP_{n}^{(\al,\nu)},\qquad \La_n^{(\al,\nu)}= n K^{(\al,\nu)}, \qquad n\in \mathbb{N}.$$
\end{cor}

Note that the matrix differential operators $\cD^{(\al,\nu)}$ of Corollary \ref{cor:second_DO_Hermite}
and $D$ of Proposition \ref{prop:symmetry_D} commute since the eigenvalue matrices commute, cf. \cite{GrunT}.

The proof of Corollary \ref{cor:second_DO_Hermite}
is completely analogous to the corresponding statement in
\cite[(2.3)]{Koe:Rio:Rom}, and we skip the proof.

\begin{rmk}
By Proposition \ref{prop:shift}(ii)  the differential operator $\mathcal{D}^{(\al,\nu)}$
has a factorization
in which first $\frac{d}{dx}$ is applied
and next $S^{(\al,\nu)}$.
The Darboux transform of the differential
operator is obtained by changing this order, and this gives a
second order matrix valued differential operator, say $\tilde{\cD}^{(\al,\nu)}$.
This operator has $P^{(\al,\nu+1)}_n$ as eigenfunctions by Proposition \ref{prop:shift};
\begin{gather*}
\Bigl( P^{(\al,\nu+1)}_n\tilde{\cD}^{(\al,\nu)}\Bigr)(x) =
\frac{d^2P^{(\al,\nu+1)}_n}{dx^2}(x) \Phi^{(\al,\nu)}(x)^\ast
+ \frac{dP^{(\al,\nu+1)}_n}{dx}(x)
\left(\frac{d\Phi^{(\al,\nu)}}{dx}(x)^\ast
+ \Psi^{(\al,\nu)}(x)^\ast \right)
\\ +  P^{(\al,\nu+1)}_n(x) \frac{d\Psi^{(\al,\nu)}}{dx}(x)^\ast =
\Xi_n^{(\al,\nu+1)} P^{(\al,\nu+1)}_n(x),
\qquad \Xi_n^{(\al,\nu+1)} = (n+1)\, \text{lc}(\Psi^{(\al,\nu)})^\ast.
\end{gather*}
Using the expression of Proposition \ref{prop:Pearson1} and \eqref{eq:lcPsialnu},
\eqref{eq:Psi0explicit} we can relate
$\Phi^{(\al,\nu)}$, $\Psi^{(\al,\nu)}$ to
$\Phi^{(\al,\nu+1)}$, $\Psi^{(\al,\nu+1)}$.
We obtain
\begin{equation*}
\begin{split}
\tilde{\cD}^{(\al,\nu)} = & \frac{d^{(\nu)}}{d^{(\nu+1)}} \cD^{(\al,\nu+1)} \\
& + d^{(\nu)}\left\{ \Bigl( \frac{d^2}{dx^2} + \frac{d}{dx} (A^\al-2x)\Bigr)
\Bigl( \frac{c^{(\nu)}}{d^{(\nu)}} - \frac{c^{(\nu+1)}}{d^{(\nu+1)}}\Bigr)
- \frac{d}{dx}A^\al + 2(J-N-1-\frac{c^{(\nu)}}{d^{(\nu)}})
\right\}
\end{split}
\end{equation*}
and for any of the examples in Sections \ref{ssec:explexample1},
\ref{ssec:explexample2}, \ref{ssec:explexample3},
the differential operator in curly brackets reduces to
the differential operator of Proposition \ref{prop:symmetry_D}
up to a multiple of the identity.
Since the eigenvalue matrices commute, these differential
operators commute, cf. \cite{GrunT}.
This is e.g. the case in \cite{Koe:Rio:Rom} for the differential
operators for the matrix valued Gegenbauer type polynomials.
In general the algebra of such differential
operators can be very complicated, see \cite{CaspY}.
\end{rmk}

Now Corollary \ref{cor:second_DO_Hermite} has the monic polynomials $P^{(\al,\nu)}_n$ as eigenfunctions.
By conjugation the differential operator of Corollary \ref{cor:second_DO_Hermite} using \eqref{eq:FLtildeF} we  find
a differential operator with $P^{(\al,\nu)}_nL^\al$ as eigenfunctions;
\begin{gather*}
\frac{d^2 P^{(\al,\nu)}_nL^\al}{dx^2}(x)
\left( \bigl( L^\al(x)\bigr)^{-1} \bigl( \Phi^{(\al,\nu)}(x)\bigr)^\ast
L^\al(x) \right)
+ \\
\frac{d P^{(\al,\nu)}_nL^\al}{dx}(x)
\left( 2\frac{d\bigl( L^\al\bigr)^{-1}}{dx}(x) \bigl( \Phi^{(\al,\nu)}(x)\bigr)^\ast
L^\al(x)  + \bigl( L^\al(x)\bigr)^{-1} \bigl( \Psi^{(\al,\nu)}(x)\bigr)^\ast
L^\al(x) \right) + \\
P^{(\al,\nu)}_nL^\al(x)
\left( \frac{d^2\bigl( L^\al\bigr)^{-1}}{dx^2}(x) \bigl( \Phi^{(\al,\nu)}(x)\bigr)^\ast
L^\al(x)  +  \frac{d\bigl( L^\al\bigr)^{-1}}{dx}(x) \bigl( \Psi^{(\al,\nu)}(x)\bigr)^\ast
L^\al(x) \right)\\ = \La^{(\al,\nu)}_n P^{(\al,\nu)}_nL^\al(x).
\end{gather*}
By Proposition \ref{prop:LisexpandLinv} we have
$\frac{d\bigl( L^\al\bigr)^{-1}}{dx}(x) = -A^\al \bigl( L^\al\bigr)^{-1}(x)$ and
$\frac{d^2\bigl( L^\al\bigr)^{-1}}{dx^2}(x) = (A^\al)^2 \bigl( L^\al\bigr)^{-1}(x)$
so that all terms can be evaluated using \eqref{eq:LinvPhistarL}
and \eqref{eq:LinvPsistarL}. This yields, after simplicification,
\begin{gather*}
\frac{d^2 P^{(\al,\nu)}_nL^\al}{dx^2}(x)
\left( d^{(\nu)}J+c^{(\nu)} \right)
+ \\
\frac{d P^{(\al,\nu)}_nL^\al}{dx}(x)
\left( -(A^\al+2x)(d^{(\nu)}J+c^{(\nu)})
+ \De^{(\nu+1)}(A^\al)^\ast (\De^{(\nu)})^{-1}\right) + \\
P^{(\al,\nu)}_nL^\al(x)
\left( 2A^\al x (d^{(\nu)}J+c^{(\nu)})
- A^\al \De^{(\nu+1)}(A^\al)^\ast (\De^{(\nu)})^{-1}\right) = \La^{(\al,\nu)}_n P^{(\al,\nu)}_nL^\al(x).
\end{gather*}
Now we take the $(r,s)$-th entry of this identity, using
$(P^{(\al,\nu)}_nL^\al)(x)_{r,s} = c^{(\al,\nu)}_{r,s}(n) H_{n+r-s}(x)$.
This gives 6 polynomial terms on the left hand side, and one on the right hand side. Considering the leading coefficient, i.e.
the coefficient of $x^{n+r-s}$, on both sides and using that
the leading coefficient of $H_m$ is $2^m$ we find
a three-term recursion for $c_{s}= c^{(\al,\nu)}_{r,s}(n)$;
\begin{equation}\label{eq:3termcoeff}
\begin{split}
&-2(n+r-s)(d^{(\nu)}s+c^{(\nu)})c_{s}
+ 4(n+r-s+1) \frac{\de^{(\nu+1)}_{s-1}}{\de^{(\nu)}_s}\frac{\al_s}{\al_{s-1}}
c_{s-1} + \\
&  2\frac{\al_{s+1}}{\al_s} (d^{(\nu)}s+c^{(\nu)}) c_{s+1} - 4 \frac{\al_s^2}{\al_{s-1}^2}\frac{\de^{(\nu+1)}_{s-1}}{\de^{(\nu)}_s}
c_{s}
= 2n\bigl( d^{(\nu)}(r-(N+1)) -c^{(\nu)}\bigr) c_{s}
\end{split}
\end{equation}
We can also obtain a three-term recurrence for the coefficients
by evaluating the differential equation at $0$, and this leads to
the same three term recurrence relation
\eqref{eq:3termcoeff}. In \eqref{eq:3termcoeff}
we eliminate the $\de^{(\nu)}_s$ using
\eqref{eq:condition_delta_Phi},
\eqref{eq:recursion-alphas}, and next we
put $c_s=\al_s^{-1} (N-s+1)_{s-1} \hat{c}_s$.
Then \eqref{eq:3termcoeff} becomes
\begin{equation}\label{eq:3termcoeffdualHahnpols}
\begin{split}
&(s+\frac{c^{(\nu)}}{d^{(\nu)}})(s-N) \hat{c}_{s+1}
- \Bigl( (s+\frac{c^{(\nu)}}{d^{(\nu)}}) (s-n+r)
+ (s-1)(s-N-1)\Bigr) \hat{c}_s  \\
&\qquad  + (s-n-r-1)(s-1) \hat{c}_{s-1} =
n(N+1-r+\frac{c^{(\nu)}}{d^{(\nu)}}) \hat{c}_s,
\qquad s\in \{1,\cdots, N\}
\end{split}
\end{equation}
and \eqref{eq:3termcoeffdualHahnpols} is
the three-term recurrence relations for the dual
Hahn polynomials, see e.g. \cite[(1.6.3)]{KoekS}, up to a rewrite of the
coefficients of $\hat{c}_s$ in \eqref{eq:3termcoeffdualHahnpols}.

\begin{thm}\label{thm:MVinHermite}
The matrix entries of the monic matrix valued
Hermite type polynomials are given by
\begin{gather*}
\bigl( P^{(\al,\nu)}_n(x)\bigr)_{r,t} =
\frac{\al_r}{2^n\al_t} \frac{1}{(r-1)!}
 \left(\prod_{k=0}^{n-1}
\frac{(1+ \frac{c^{(\nu+k)}}{d^{(\nu+k)}})}{(N-r+1+ \frac{c^{(\nu+k)}}{d^{(\nu+k)}})} \right) \\
\times \sum_{s=t}^{N \wedge (n+r)} \frac{i^{s-t} (N-s+1)_{s-1}}{(s-t)!}
\, \rFs{3}{2}{1-s, r-N, n+1+ c^{(\nu)}/d^{(\nu)}}
{1-N,\, 1+ c^{(\nu)}/d^{(\nu)}}{1}
 H_{n+r-s}(x)
 H_{s-t}(ix)
\end{gather*}
\end{thm}

Theorem \ref{thm:MVinHermite} is the analogue of \cite[Thm.~3.4]{Koe:Rio:Rom}, in which
the result of Cagliero and Koornwinder \cite{CaglK} plays an essential role.

\begin{rmk}\label{rmk:thm:MVinHermite}
(i) With the explicit evaluation of  the coefficients $c^{(\al,\nu)}_{r,s}$ in the proof
of Theorem \ref{thm:MVinHermite}, see
\eqref{eq:crsnas3F2} and \eqref{eq:crnsinitialvalue}, we can write out the
orthogonality relations for $P_n^{(\al,\nu)}L^\al$ with respect the diagonal weight
$e^{-x^2}\De^{(\nu)}$. Then the orthogonality for the degree and the fact that
the squared norm $H^{(\al,\nu)}_n$ is diagonal, see Proposition \ref{prop:shift}(iv),
follows from the orthogonality of the Hermite polynomials and the orthogonality of
the Hahn polynomials, see \cite[\S 1.5]{KoekS}. \par\noindent
(ii) The case $n=0$ gives $P_0(x)=I$, since the ${}_3F_2$-series
reduces to a terminating ${}_2F_1$-series summable by the Chu-Vandermonde sum. The
resulting sum is precisely \eqref{eq:Hermiteinverse}. More general, if $n<N$
then for certain matrix entries, the sum terminates at $N+r$. In this case,
a transformation formula for ${}_3F_2$-series, see e.g. \cite[p.~142]{AndrAR}, can be
used.
\par\noindent
(iii) Note that the left hand side is of degree at most $n$, whereas the right hand side
initially is of degree $n+r-t$, which can be larger. So these sums give zero, e.g.
by considering the leading coefficient of the right hand side for $r>t$ we find
\[
\sum_{s=t}^{N \wedge (n+r)} \frac{(-1)^{s-t} (N-s+1)_{s-1}}{(s-t)!}
\, \rFs{3}{2}{1-s, r-N, n+1+ c^{(\nu)}/d^{(\nu)}}
{1-N,\, 1+ c^{(\nu)}/d^{(\nu)}}{1} =0.
\]
\end{rmk}

\begin{proof}[Proof of Theorem \ref{thm:MVinHermite}]
Comparing the recurrence relation \eqref{eq:3termcoeffdualHahnpols}
with \cite[(1.6.3)]{KoekS} in which we replace $(\ga,\de,N,n,x)$
by $(\frac{c^{(\nu)}}{d^{(\nu)}},n+r-N,N-1,s-1,N-r)$
gives
\begin{equation}\label{eq:crsnas3F2}
c^{(\al,\nu)}_{r,s}(n) = c^{(\al,\nu)}_{r,1}(n)
\al_s^{-1} (N-s+1)_{s-1}
\, \rFs{3}{2}{1-s, r-N, n+1+ c^{(\nu)}/d^{(\nu)}}
{1-N,\, 1+ c^{(\nu)}/d^{(\nu)}}{1}
\end{equation}
and by differentiating $(P^{(\al,\nu)}_nL^\al)_{r,1}(x)
= c^{(\al,\nu)}_{r,1}(n) H_{n+r-1}$ and using Proposition \ref{prop:shift},
Proposition \ref{prop:shift}(ii) we
get
\[
2(n+r-1) c^{(\al,\nu)}_{r,1}(n) = n
c^{(\al,\nu+1)}_{r,1}(n-1) +2\al_2
c^{(\al,\nu)}_{r,2}(n).
\]
Since we already have
\[
 c^{(\al,\nu)}_{r,2}(n) =
 \frac{-\bigl( (N-r)n +(1-r)(1+ c^{(\nu)}/d^{(\nu)})\bigr)}
 {\al_2(1+ c^{(\nu)}/d^{(\nu)})}
 c^{(\al,\nu)}_{r,1}(n)
\]
by expanding the ${}_3F_2$-series, we get a
simple recursion. After simplifying
\begin{gather}
2\frac{(N-r+1+ c^{(\nu)}/d^{(\nu)})}{(1+ c^{(\nu)}/d^{(\nu)})}
  c^{(\al,\nu)}_{r,1}(n)
=   c^{(\al,\nu+1)}_{r,1}(n-1)\quad
\Longrightarrow  \nonumber\\
c^{(\al,\nu)}_{r,1}(n) =
c_{r,1}^{(\al,\nu+n)}(0)\,\, 2^{-n}
\prod_{k=0}^{n-1}
\frac{(1+ c^{(\nu+k)}/d^{(\nu+k)})}{(N-r+1+ c^{(\nu+k)}/d^{(\nu+k)})}, \label{eq:crnsinitialvalue}
\end{gather}
and since we have
the initial value $c^{(\al,\nu+n)}_{r,1}(0)=\al_r/(r-1)!$,
we have an explicit expression for
$P^{(\nu,\al)}_n(x) L^\al(x) =
c^{(\al,\nu)}_{r,s}(n) H_{n+r-s}(x)$.
The result follows from
\[
P^{(\nu,\al)}_n(x)_{r,t} =
\sum_{s=1}^N \bigl(P^{(\nu,\al)}_n(x)L^\al(x)\bigr)_{r,s} \bigl( L^\al(x)\bigr)^{-1}_{s,t}
= \sum_{s=t}^N
c^{(\al,\nu)}_{r,s}(n) H_{n+r-s}(x) \frac{\al_s}{\al_t}
i^{s-t} \frac{H_{s-t}(ix)}{(s-t)!}
\]
by Lemma \ref{lem:Lm1JL}.
\end{proof}

As an application of Theorem \ref{thm:MVinHermite} we derive a connection formula
for the matrix valued Hermite type polynomials. Needless to say that this has no
scalar analogue, since the Hermite polynomials have no degree of freedom as they are at the bottom
of the Askey scheme.
In order to prove the connection formula, we require a connection formula for Hahn
polynomials due to Gasper \cite{Gasp-JMAA}. Take $a=\al$ in \cite[(1.4), (4.1)]{Gasp-JMAA},
so that the ${}_3F_2$-series in \cite[(4.1)]{Gasp-JMAA} reduces to a summable ${}_2F_1$-series.
Reversing the sum in \cite[(4.1)]{Gasp-JMAA} gives the special case of Gasper's
identity;
\begin{equation}\label{eq:Hahnconnectioncoef}
\begin{split}
Q_n(x;a,b,N)=\sum_{k=0}^n &(-1)^p \binom{n}{p} \frac{ (b-\be)_p (n+a+b+1)_{n-p}}{(n-p+a+\be+1)_n}
\\ &\times \frac{(2n-2p+a+\be+1)}{(2n-p+a+\be+1)}
Q_{n-p}(x;a,\be,N).
\end{split}
\end{equation}
Using $(a+1)_n Q_n(x;a,b,N)= (-1)^n(b+1)_n Q_n(N-x;b,a,N)$, see e.g.
\cite[(2.7)]{Gasp-JMAA} in \eqref{eq:Hahnconnectioncoef}
we get the connection coefficient formula for Hahn polynomials with the second parameter equal and
with different first parameter. Calling $y=N-x$, we get an expansion relation for ${}_3F_2$-series
arising in Theorem \ref{thm:MVinHermite}. Upon replacing $y\mapsto s-1$, $N\mapsto N-1$, $n\mapsto N-r$,
$a\mapsto r-N+n$, $b\mapsto c^{(\nu)}/d^{(\nu)}$, $\be\mapsto c^{(\la)}/d^{(\la)}$ we get
\begin{gather}
\rFs{3}{2}{1-s, r-N, n+1+ c^{(\nu)}/d^{(\nu)}}
{1-N,\, 1+ c^{(\nu)}/d^{(\nu)}}{1}  =
\sum_{p=0}^{N-r} \binom{N-r}{p} (\frac{c^{(\nu)}}{d^{(\nu)}}-\frac{c^{(\la)}}{d^{(\la)}})_p
\frac{(\frac{c^{(\la)}}{d^{(\la)}}+1)_{N-r-p}}{(\frac{c^{(\nu)}}{d^{(\nu)}}+1)_{N-r-p}}
\nonumber  \\
\times  \frac{(N-r-2p+n+ 1+ \frac{c^{(\la)}}{d^{(\la)}})}{(N-r-p+n+ 1+ \frac{c^{(\la)}}{d^{(\la)}})}
\frac{(-1)^p (-n)_p (n+1+\frac{c^{(\nu)}}{d^{(\nu)}})_{N-r-p}}{(p+n+1+\frac{c^{(\la)}}{d^{(\la)}})_{N-r}}
\nonumber \\
\times
\rFs{3}{2}{1-s, r-N+p, n-p+1+ c^{(\la)}/d^{(\la)}}
{1-N,\, 1+ c^{(\la)}/d^{(\la)}}{1}. \label{eq:explconnection}
\end{gather}
This is the key identity for the connection formula in Proposition \ref{prop:Hermiteconnection}.

\begin{prop}\label{prop:Hermiteconnection}
The matrix valued Hermite type polynomials satisfy
$$
P_n^{(\al,\nu)}(x) = \sum_{k=0}^N A^{(\nu,\la)}_k(n) P^{(\al,\la)}_{n-k}(x),
$$
where the matrix entry $A^{(\nu,\la)}_k(n)_{r,t}=0$ unless $t=r+k$ and
\begin{multline*}
A^{(\nu,\la)}_k(n)_{r,r+k}= 2^{-k}
\prod_{p=0}^{n-1} \frac{(1+\frac{c^{(\nu+p)}}{d^{(\nu+p)}})}{(N-r+1+\frac{c^{(\nu+p)}}{d^{(\nu+p)}})}
\prod_{p=0}^{n-k-1} \frac{(N-r+1+\frac{c^{(\la+p)}}{d^{(\la+p)}})}{(1+\frac{c^{(\la+p)}}{d^{(\la+p)}})} \\
\times \binom{N-r}{k} (\frac{c^{(\nu)}}{d^{(\nu)}}-\frac{c^{(\la)}}{d^{(\la)}})_k
\frac{(\frac{c^{(\la)}}{d^{(\la)}}+1)_{N-r-k}}{(\frac{c^{(\nu)}}{d^{(\nu)}}+1)_{N-r-k}}
\frac{(N-r-2k+n+ 1+ \frac{c^{(\la)}}{d^{(\la)}})}{(N-r-k+n+ 1+ \frac{c^{(\la)}}{d^{(\la)}})} \\
\times \frac{(-1)^k (-n)_k (n+1+\frac{c^{(\nu)}}{d^{(\nu)}})_{N-r-k}}{(k+n+1+\frac{c^{(\la)}}{d^{(\la)}})_{N-r}}
\end{multline*}
\end{prop}

The general connection formula of Proposition \ref{prop:Hermiteconnection} simplifies
for the special cases of Sections \ref{ssec:explexample1}, \ref{ssec:explexample2}, \ref{ssec:explexample3}.
Indeed, in the special cases of Sections \ref{ssec:explexample1}, \ref{ssec:explexample2} we find
\begin{gather}
A^{(\nu,\la)}_k(n)_{r,r+k}=
\binom{N-r}{k}  \frac{(-2)^{-k}\, (\la-\nu)_k  (-n)_k}{(k+n+1+\la)_{N-r}}
\frac{(n+1+\nu)_n}{(N-r+1+\nu)_n}
\frac{(n-k+1+\la)_{N-r}}{(N-r-k+1+\la)_{k}} \nonumber \\
\times \frac{(N-r-2k+n+ 1+ \la)}{(N-r-k+n+ 1+ \la)}. \label{eq:connectioncoefspecialcase}
\end{gather}
The special case of Section \ref{ssec:explexample3} can also be simplified
as in \eqref{eq:connectioncoefspecialcase}. In \eqref{eq:connectioncoefspecialcase}
replace $\la$ and $\nu$ by $\la+C\rho^{-1}$ and $\nu+C\rho^{-1}$ to get the
corresponding expression for the special case of Section \ref{ssec:explexample3}.

\begin{proof} Instead of proving the identity, we multiply by the invertible matrix $L^{\al}(x)$.
Taking the $(r,s)$-matrix entry, using $A^{(\nu,\la)}_k(n)_{r,t}=0$ for $t\not=r+k$ and Theorem \ref{thm:MVinHermite} we get
an identity where on both sides $H_{n+r-s}(x)$ can be cancelled. The identity to be proved is
\[
c^{(\al,\nu)}_{r,s}(n) = \sum_{k=0}^{N-r} A^{(\nu,\la)}_k(n)_{r,r+k}\,
c^{(\al,\la)}_{r+k,s}(n-k)
\]
and using the explicit expression for the coefficients from Theorem \ref{thm:MVinHermite} and its
proof, we see that we need to show that, after canceling common factors
\begin{gather*}
\rFs{3}{2}{1-s, r-N, n+1+ \frac{c^{(\nu)}}{d^{(\nu)}}}
{1-N,\, 1+ \frac{c^{(\nu)}}{d^{(\nu)}}}{1}
\prod_{p=0}^{n-1} \frac{(1+\frac{c^{(\nu+p)}}{d^{(\nu+p)}})}{(N-r+1+\frac{c^{(\nu+p)}}{d^{(\nu+p)}})} = \\
\sum_{k=0}^{N-r} A^{(\nu,\la)}_k(n)_{r,r+k}\, 2^k
\prod_{p=0}^{n-k-1} \frac{(1+\frac{c^{(\la+p)}}{d^{(\la+p)}})}{(N-r+1+\frac{c^{(\la+p)}}{d^{(\la+p)}})}
\rFs{3}{2}{1-s, r-N+k, n-k+1+ \frac{c^{(\la)}}{d^{(\la)}}}
{1-N,\, 1+ \frac{c^{(\la)}}{d^{(\la)}}}{1}.
\end{gather*}
Comparing with \eqref{eq:explconnection} gives the result.
\end{proof}


\subsection{The three-term recurrence relation}
We assume the condition of Proposition \ref{prop:shift}
to be valid.
The monic matrix valued
Hermite type orthogonal polynomials satisfy a three-term recurrence relation of the form, cf.
Section \ref{ssec:MVOPS},
\begin{equation}
\label{eq:recurrence_Hermite_general}
xP^{(\al,\nu)}_{n}(x)=P_{n+1}^{(\al,\nu)}(x)+B^{(\al,\nu)}_{n}P_{n}^{(\al,\nu)}(x)+C_{n}^{(\al,\nu)}P_{n-1}^{(\al,\nu)}(x)
\end{equation}
In order to calculate the coefficients $B^{(\al,\nu)}_{n}$ and $C^{(\al,\nu)}_{n}$, we use the fact that
\begin{equation}\label{eq:Bnin1blc}
B_{n}^{(\al,\nu)} = X_{n}^{(\al,\nu)}-X^{(\al,\nu)}_{n+1},\qquad C_{n}^{(\al,\nu)} = H_{n}^{(\al,\nu)}(H_{n-1}^{(\al,\nu)})^{-1},
\end{equation}
where  $X_{n}^{(\al,\nu)}$ is the one-but-leading coefficient of $P_{n}^{(\al,\nu)}$, i.e. $P_{n}^{(\al,\nu)}=x^nI+x^{n-1}X_{n}^{(\al, \nu)}+\cdots$.
If we differentiate $P_{n}^{(\al,\nu)}$ with respect to $x$ and we use Proposition \ref{prop:shift} (ii), we find that
$(n-1)X_{n}^{(\al,\nu)}=nX^{(\al,\nu+1)}_{n+1}$ which gives $X_{n}^{(\al,\nu)}=nX_{1}^{(\al,\nu+n-1)}$, so that it is sufficient to check the polynomials of degree one. Using the Rodrigues formula
of Proposition \ref{prop:shift}(iii) we obtain $P_{1}^{(\al,\nu)}(x)=G_{1}^{(\al,\nu)}(\Psi^{(\al,\nu)}(x))^\ast$. Therefore, we have that
\begin{equation}\label{eq:1blcasPsi}
X_{n}^{(\al,\nu)} =  n\,(K^{(\al,\nu+n-1)} )^{-1}
\bigl( \Psi^{(\al,\nu+n-1)}(0)\bigr)^\ast.
\end{equation}

\begin{prop}\label{prop:3termrecur}
Assume the conditions of Proposition \ref{prop:shift}.  Using the notation as in Proposition \ref{prop:shift}, the coefficients of the three-term recurrence relation
\eqref{eq:recurrence_Hermite_general} for the monic Hermite-type polynomials are given by
\begin{align*}
B_{n}^{(\al,\nu)}&= \frac12 A^\al  +
\frac14
\left( N+1-J+ \frac{c^{(\nu+n-1)}}{d^{(\nu+n-1)}}\right)^{-1}
\left( N+1-J+ \frac{c^{(\nu+n)}}{d^{(\nu+n)}}\right)^{-1}
\\
\times &
\left( N+1-J +\Bigl(
(n+1)\frac{c^{(\nu+n-1)}}{d^{(\nu+n-1)}}
- n\frac{c^{(\nu+n)}}{d^{(\nu+n)}}\Bigr) \right)  J(N-J) S^\al A^\ast (S^\al)^{-1}\\
\bigl(C_{n}^{(\al,\nu)}\bigr)_{p,q}&=
\de_{p,q} \frac{-2n}{d^{(\nu+n-1)}}
\frac{\de_1^{(\nu+n)}}{\de_1^{(\nu+n-1)}}
\frac{(-N- \frac{c^{(\nu+n)}}{d^{(\nu+n)}})_{p-1}}
{(-N- \frac{c^{(\nu+n-1)}}{d^{(\nu+n-1)}})_{p}}.
\end{align*}
\end{prop}

Note that $C_{n}^{(\al,\nu)}$ is diagonal and positive definite, and
that $B_{n}^{(\al,\nu)}$ only has a non-zero sub- and superdiagonal.

\begin{proof}
We have
$C_{n}^{(\al,\nu)} = H_{n}^{(\al,\nu)}(H_{n-1}^{(\al,\nu)})^{-1}$, cf. \eqref{eq:Hn_CnHnm1},
and plugging in the expression of
Proposition \ref{prop:shift}(iv), taking into account
that all terms commute, the result follows from the explicit
expressions for $K^{(\al,\nu)}$ and $H^{(\al,\nu)}_0$ in
\eqref{eq:lcPsialnu} and Corollary \ref{cor:lem:explWalnumn}.

We first consider $B_0^{(\al,\nu)}$. From the three-term recurrence
we obtain an expression for $P^{(\al,\nu)}_1$, which we compare
with Proposition \ref{prop:shift}(iii), giving
$B^{(\al,\nu)}_0 = - \text{lc}(\Psi^{(\al,\nu})^{-1} \Psi^{(\al,\nu)}(0)^\ast$,
since the leading coefficient is self-adjoint.
By construction $B^{(\al,\nu)}_0H^{(\al,\nu)}_0$ is self-adjoint, see Section \ref{ssec:MVOPS}.
Using the explicit expression, we find
\begin{equation}\label{eq:Psi0starandPsi0}
\Psi^{(\al,\nu)}(0) = \left( \text{lc}(\Psi^{(\al,\nu}))^{-1}
H^{(\al,\nu)}_0 \right)^{-1} \Psi^{(\al,\nu)}(0)^\ast\,
\text{lc}(\Psi^{(\al,\nu}))^{-1}
H^{(\al,\nu)}_0.
\end{equation}
Since the conjugation in \eqref{eq:Psi0starandPsi0} is by a diagonal matrix, and $\Psi^{(\al,\nu)}(0)$
is a Hessenberg matrix, we have that $\Psi^{(\al,\nu)}(0)$
is tridiagonal. The diagonal being zero is already clear from
\eqref{eq:lcPsialnu}. This shows that from
\eqref{eq:Bnin1blc} and \eqref{eq:1blcasPsi} we have $B^{(\al,\nu)}$
as a tridiagonal matrix with zero diagonal.

From \eqref{eq:lcPsialnu} we have
\[
 \Psi^{(\al,\nu)}(0)_{i,i-1}
= 2\frac{\al_i}{\al_{i-1}}\frac{\de^{(\nu+1)}_{i-1}}{\de^{(\nu)}_{i}}
= \frac{\al_{i-1}}{\al_i} d^{(\nu)} (i-1)(N-i+1),
\]
where the second equality follows from \eqref{eq:recursion-alphas},
and using \eqref{eq:Psi0starandPsi0}, we can obtain
\[
\Psi^{(\al,\nu)}(0)_{i-1,i}
= \frac{H^{(\al,\nu)}_{i,i}}{H^{(\al,\nu)}_{i-1,i-1}}
\frac{\text{lc}\Psi^{(\al,\nu)}_{i,i}}{\text{lc}\Psi^{(\al,\nu)}_{i-1,i-1}}
2\frac{\al_i}{\al_{i-1}}\frac{\de^{(\nu+1)}_{i-1}}{\de^{(\nu)}_{i}}
= 2\frac{\al_i}{\al_{i-1}}\bigl(c^{(\nu)}+ d^{(\nu)}(N+1-i)\bigr)
\]
by Corollary \ref{cor:lem:explWalnumn}, \eqref{eq:lcPsialnu}
and again \eqref{eq:recursion-alphas}. Combining gives
\begin{equation}\label{eq:Psi0explicit}
\Psi^{(\al,\nu)}(0) = (A^\al)^\ast
\bigl(c^{(\nu)}+ d^{(\nu)}(N+1-J)\bigr)
+
\frac12 d^{(\nu)} (S^\al)^{-1} A S^\al J(N-J).
\end{equation}
Using this expression in \eqref{eq:Bnin1blc} using
\eqref{eq:1blcasPsi} gives the result by a straightforward calculation.
\end{proof}


\subsection{Irreducibility of the weight}
We assume that the conditions of
Proposition \ref{prop:shift} hold, so that
the matrix valued orthogonal polynomials
exist, and that all the required properties
are valid.

For the general weight $W^{(\al,\nu)}$, let $T\in\cA$, then by \cite[Lemma~3.1]{KR15}, we
have $[T,C^{(\al,\nu)}_n]=0$ for all $n\in \N$, see
\eqref{eq:commutationBnHn} and Section \ref{ssec:MVOPS}.
Recall from Proposition \ref{prop:3termrecur} that $C^{(\al,\nu)}_n$ is diagonal.
First observe that, assuming
without loss of generality that $p\geq r$,
we have
\[
(C^{(\al,\nu)}_n)_{p,p} = (C^{(\al,\nu)}_n)_{r,r}
\quad \Longleftrightarrow \quad
(-N+r-1-\frac{c^{(\nu+n)}}{d^{(\nu+n)}})_{p-r} =
(-N+r-\frac{c^{(\nu+n-1)}}{d^{(\nu+n-1)}})_{p-r}
\]
and since all the terms in the shifted factorials are negative for all possible choices, this
implies $1+\frac{c^{(\nu+n)}}{d^{(\nu+n)}}=
\frac{c^{(\nu+n-1)}}{d^{(\nu+n-1)}}$ or $p=r$.
If $T$ has a non-zero off-diagonal entry, say $T_{p,r}\not=0$, $p\not=r$, we need
$(C^{(\al,\nu)}_n)_{p,p} = (C^{(\al,\nu)}_n)_{r,r}$
for all $n\in \N$, hence
$1+\frac{c^{(\nu+n)}}{d^{(\nu+n)}}=
\frac{c^{(\nu+n-1)}}{d^{(\nu+n-1)}}$ for all $n\in \N$. Or $\frac{c^{(\nu+n)}}{d^{(\nu+n)}}
= -n + \frac{c^{(\nu)}}{d^{(\nu)}}$ for all $n\in \N$, contradicting the positivity assumption on
$c^{(\nu+n)}$ and $d^{(\nu+n)}$ for all $n\in \N$
as in Proposition \ref{prop:shift}.

\begin{prop}\label{prop:irreducibility}
Consider the algebra $\cA$ for the weight
$W^{(\al, \nu)}$ under the assumptions as
in Proposition \ref{prop:shift}.
Then $\cA$ is trivial, and the weight is
irreducible.
\end{prop}

\begin{proof}
The discussion before Proposition \ref{prop:irreducibility} shows that any $T\in\cA$
has to be diagonal.
By \cite[Lemma~3.1]{KR15} $[T, B^{(\al,\nu)}_n]=0$,
for all $n\in \N$, so that in particular
$[T, B^{(\al,\nu)}_n-B^{(\al,\nu)}_{n-1}]=0$.
By Proposition \ref{prop:3termrecur}
we see that $B^{(\al,\nu)}_n-B^{(\al,\nu)}_{n-1}$
consists of a product of diagonal matrices and
$A^\ast$. Since we have $T$ as a diagonal matrix,
the commutation shows that $[T,A^\ast]=0$.
This means that all diagonal elements of $T$
are equal, and the result follows.
\end{proof}


\subsection{Explicit matrix valued Hermite polynomials}\label{ssec:explexample1}

In general we cannot find all solutions to
\eqref{eq:condition_delta_Phi} and \eqref{eq:recursion-alphas}, since the relations are non-linear.
Starting with \eqref{eq:recursion-alphas}
we assume now that the $\nu$-independent left hand side
coincides with $(N-k)/2$ in the left hand side.
So we take
\begin{equation}\label{eq:alpha-example1}
\al_k=\sqrt{2^{1-k} (N-k+1)_{k-1}}, \qquad k=1,\ldots,N.
\end{equation}
Then \eqref{eq:condition_delta_Phi} and \eqref{eq:recursion-alphas} give the following recurrence relations for the coefficients $\de_k^{(\nu)}$:
\begin{equation}
\label{eq:rec_delta_example1}
\de_k^{(\nu+1)}=(d^{(\nu)}k+c^{(\nu)})\de_k^{(\nu)},\qquad \de_{k+1}^{(\nu)} = \left(\frac{d^{(\nu)}k+c^{(\nu)}}{d^{(\nu)}k}\right) \de_{k}^{(\nu)}.
\end{equation}
A possible solution to \eqref{eq:rec_delta_example1} is given by
$$
\de_k^{(\nu)} =\frac{(\nu+1)_{k-1}}{(k-1)!},\qquad c^{(\nu)} = \frac{\nu}{\nu+1}, \qquad d^{(\nu)}=\frac{1}{\nu+1}
$$
and assuming $\nu>0$ we have that $\De^{(\nu)}$ is positive
definite and $c^{(\nu)}$ and $d^{(\nu)}$ positive, so that
all conditions, and in particular \eqref{eq:condition_delta_Phi} and \eqref{eq:recursion-alphas} are satisfied.

The general results for certain coefficients, such as the expressions in Proposition \ref{prop:3termrecur}
and Theorem \ref{thm:MVinHermite}, will simplify slightly using the explicit expressions
since $c^{(\nu)}/d^{(\nu)}=\nu$.


\subsection{Matrix valued Hermite polynomials related to $\mathfrak{sl}(2)$}
\label{ssec:explexample2}
Take $\al_{k+1}^2/\al_k^2=k(N-k)/4$, then
 $\al_k=2^{1-k} \sqrt{ (k-1)! (N-k+1)_{k-1}}$, $k=1,\ldots,N$. Then \eqref{eq:condition_delta_Phi} and \eqref{eq:recursion-alphas} give the following recurrence relations for the coefficients $\de_k^{(\nu)}$:
\begin{equation}
\label{eq:rec_delta_example2}
\de_k^{(\nu+1)}=(d^{(\nu)}k+c^{(\nu)})\de_k^{(\nu)},\qquad \de_{k+1}^{(\nu)} = \frac12 \left(k+\frac{c^{(\nu)}}{d^{(\nu)}}\right) \de_{k}^{(\nu)}.
\end{equation}
Assume $c^{(\nu)}/d^{(\nu)}=\nu$, then the solution to first equation is
easily obtained. Putting, $c^{(\nu)}=\la \nu$, $d^{(\nu)} =\la$ we
find a solution to \eqref{eq:rec_delta_example2} by
$$
\de_k^{(\nu)} =2^{-k}\, \la^{\nu}\, \Ga(\nu+k)=
2^{-k}\, (\nu)_k \, \la^{\nu}\, \Ga(\nu)
,\qquad c^{(\nu)} = \nu\la, \qquad d^{(\nu)}=\la
$$
for some fixed $\la>0$. The positivity requirements require $\nu>0$.
Then \eqref{eq:condition_delta_Phi} and \eqref{eq:recursion-alphas} are satisfied.

Since now
$c^{(\nu)}/d^{(\nu)}=\nu$ as well, most simplifications of
Section \ref{ssec:explexample1} will hold in this setting as well.

Up to relabeling the case $\nu_i = \sqrt{2i}$ of Dur\'an \cite{Dura-CA} fits
in this family with a suitable choice.

\begin{rmk} Let $E,F,H$ with relations $[H,E]=2E$, $[H,F]=-2F$ and
$[E,F]=H$ give the Lie algebra $\mathfrak{sl}(2,\C)$. Suppose that
$\C^N$ has orthonormal basis $\{e_1,\cdots, e_N\}$, then the finite-dimensional representation $\pi\colon \mathfrak{sl}(2) \to \text{End}(\C^N)$ is given
by
\begin{gather*}
\pi(E) e_k = \sqrt{(k-1)(N+1-k)} e_{k-1}, \quad
\pi(F) e_k = \sqrt{k(N-k)} e_{k+1}, \\
\pi(H) e_k = (N+1-2k) e_k.
\end{gather*}
This is an irreducible finite-dimensional unitary resentation for the $\ast$-structure
$E^\ast=F$, $H^\ast=H$ corresponding to the real form $\mathfrak{su}(2)$
of $\mathfrak{sl}(2,\C)$. Then $A^\al = \pi(F)$, $(A^\al)^\ast = \pi(E)$,
and $J= \frac12(N+1 -\pi(H))$, so that $e^{xA^\al}= \pi(\exp(xF))$ where
we take the corresponding representation of $SU(2)$ and we consider
$\exp(xF)\in SU(2)$. Some identities used, e.g. \eqref{eq:conjAHofJ}
and Lemma \ref{lem:Lm1JL},
can now be obtained directly from the Lie group interpretation. However, it is
not clear if this relation can be extended further.
\end{rmk}


\subsection{Another example of explicit matrix valued Hermite polynomials}
\label{ssec:explexample3}
We now take $\al_k=1$ to be constant, so that we need to solve
\begin{equation*}
1=\frac{d^{(\nu)}k(N-k)}{2(d^{(\nu)}k+c^{(\nu)})} \frac{\de^{(\nu)}_{k+1}}{\de^{(\nu)}_{k}}, \qquad \de^{(\nu+1)}_{k} = (d^{(\nu)}k+c^{(\nu)}) \de^{(\nu)}_{k+1}
\end{equation*}
for which
\begin{equation}
d^{(\nu)}=\rho, \qquad c^{(\nu)} = C+ \nu\rho,
\qquad \de^{(\nu)}_{k} = \frac{2^{k-1} (1+\nu + C/\rho)_{k-1}}{(k-1)!\, (N-k+1)_{k-1}} \rho^\nu \,\Ga(1+\nu +C/\rho)
\end{equation}
with $\rho>0$, $\nu>0$ and $C\geq 0$ gives a solution meeting all the conditions.

Since now
$c^{(\nu)}/d^{(\nu)}=\nu +C\rho^{-1}$ the simplifications of
will be somewhat more general.


\section{A Burchnall-type formula for matrix valued orthogonal polynomials}
\label{sec:Burchnall_formula}

The matrix valued Hermite polynomials of Section
\ref{sec:MV_Hermite} share many properties
with the matrix valued Gegenbauer type
polynomials previously introduced in
\cite{Koe:Rio:Rom}. In particular, the
shift properties and the Rodrigues expression
hold in both cases, see Proposition \ref{prop:shift} and \cite[Thm.~3.1]{Koe:Rio:Rom}.
This is the essential ingredient to
push the Burchnall-type formulas of
\cite{IKR} to the matrix level. In contrast with \cite{IKR}, we only
need to deal with $\frac{d}{dx}$ as lowering operator, since
these are the only examples available now.
The proofs of the statements in Section \ref{sec:Burchnall_formula} are analogous to the ones
in the scalar case as given in in \cite[\S 2]{IKR}.

Let us assume we have a matrix weight $W^{(\nu)}$ for some parameter
space $\nu\in \cV$
supported in the interval $[a,b]$, where $a$ and $b$ are allowed to be infinite.
We assume that with $\nu\in \cV$ we also
have $\nu+1\in \cV$.
We assume the existence of matrix valued polynomials
$\Phi^{(\nu)}$ and $\Psi^{(\nu)}$ of degree (at most) two and one respectively, such that
we have the Pearson equations
\begin{equation}
\label{eq:Pearson_formal}
W^{(\nu+1)}(x)=W^{(\nu)}(x)\Phi^{(\nu)}(x),\qquad \frac{dW^{(\nu+1)}}{dx}(x)=W^{(\nu)}(x)\Psi^{(\nu)}(x),
\end{equation}
for all $x\in[a,b]$.
Moreover, we assume that there exist shift operators $\frac{d}{dx}$ and $S^{(\nu)}=(\frac{d}{dx})(\Phi^{(\nu)}(x))^\ast +(\Psi^{(\nu)}(x))^\ast ,$ such that
\begin{equation}
\label{eq:shift_formal}
\frac{dP^{(\nu)}_n}{dx}(x)=n\, P_{n-1}^{(\nu+1)}(x), \qquad (P^{(\nu+1)}_{n-1}S^{(\nu)})(x)=K^{(\nu)}_nP_{n}^{(\nu)}(x),
\end{equation}
where $K_n^{(\nu)}$ is a constant invertible matrix and
$\{P^{(\nu)}_n\}_{n\geq0}$ is the sequence of monic orthogonal polynomials with respect to $W^{(\nu)}$.
Note that
\begin{equation}
\label{eq:rasing_on_Q}
(QS^{(\nu)})(x)=\left[ \frac{d}{dx} \left( Q(x)W^{(\nu+1)}(x) \right)\right] \left(W^{(\nu)}(x)\right)^{-1}.
\end{equation}
Moreover we assume that, for every $n\in \N$, the polynomial $P_n^{(\nu)}$ is given by the Rodrigues formula
obtained by iterating \eqref{eq:shift_formal};
\begin{equation}
\label{eq:Rodrigues_formal}
P^{(\nu)}_n(x)=G^{(\nu)}_n \left( \frac{d^nW^{(\nu+n)}}{dx^n}(x)\right) \left(W^{(\nu)}(x)\right)^{-1}
= (\cdots ((\textbf{1}S^{(\nu+n-1)})S^{(\nu+n-2)})\cdots S^{(\nu)})(x),
\end{equation}
where $\textbf{1}$ is the constant function, $G^{(\nu)}_n
= (K_n^{(\nu)})^{-1}\cdots (K_1^{(\nu+n-1)})^{-1}$ is a constant matrix. Note that this in particular
requires the weight to decay sufficiently fast at the endpoints.

The typical example are the Hermite-type polynomials introduced Section \ref{sec:MV_Hermite}, see in
particular Propositions \ref{prop:Pearson1}, \ref{prop:Pearson2},
\ref{prop:shift},
and the Gegenbauer-type polynomials given in \cite{Koe:Rio:Rom}, where the matrices $K_n^{(\nu)}$ and $G^{(\nu)}_n$ and
the polynomials $\Phi^{(\nu)}$ and $\Psi^{(\nu)}$ are
calculated explicitly
see \cite[Thm.~3.1, (4.9), (4.10)]{Koe:Rio:Rom}.

Now we can obtain the matrix valued analogue of
\cite[Thm.~2.1]{IKR}.

\begin{thm}[Burchnall's formula]
\label{thm:Burchnall_general}
Assume the general conditions of Section \ref{sec:Burchnall_formula}.
Let $Q$ be a $C^\infty$ matrix valued function. Then we have
\begin{gather*}
\label{eq:Matrix_Burchnall}
\Bigl(QS^{(\nu+n-1)} \cdots S^{(\nu+1)} S^{(\nu)}\Bigr)(x) = \\
\sum_{k=0}^n \binom{n}{k} Q^{(k)}(x)  \bigl( G_{n-k}^{(\nu+k)}\bigr)^{-1}
P_{n-k}^{(\nu+k)}(x)  \left(\Phi^{(\nu)}(x) \cdots \Phi^{(\nu+k-1)}(x)\right)^\ast .
\end{gather*}
\end{thm}

\begin{proof}
Iterating \eqref{eq:rasing_on_Q} we find that the left hand side of \eqref{eq:Rodrigues_formal} gives
\begin{multline*}
\Bigl( \bigl( (QS^{(\nu+n-1)}) \cdots S^{(\nu+1)}\bigr)S^{(\nu)}\Bigr)(x) =
\left[ \frac{d^n}{dx^n} \left( Q(x)W^{(\nu+n)}(x) \right)\right] \left(W^{(\nu)}(x)\right)^{-1} \\
=
\sum_{k=0}^n \binom{n}{k} Q^{(k)}(x)  \frac{d^{n-k}W^{(\nu+n)}}{dx^{n-k}}(x) \left( W^{(\nu)}(x)\right)^{-1}\\
=\sum_{k=0}^n \binom{n}{k} Q^{(k)}(x)  \bigl( G_{n-k}^{(\nu+k)}\bigr)^{-1}
P_{n-k}^{(\nu+k)}(x) W^{(\nu+k)}(x) \left( W^{(\nu)}(x)\right)^{-1}
\end{multline*}
using the Rodrigues formula \eqref{eq:Rodrigues_formal} and the
Leibniz rule.
Now we iterate \eqref{eq:Pearson_formal}
and take its adjoint to find the result.
\end{proof}

If we replace $Q=
(\cdots ((\textbf{1}S^{(\nu+n-1)})S^{(\nu+n-2)})\cdots S^{(\nu)})(x) = (G^{(\nu+n)}_m)^{-1}
P_m^{(\nu+n)}$ in Theorem \ref{thm:Burchnall_general}, and we use the shift operators \eqref{eq:shift_formal}, we obtain the an extension of \cite[(5)]{Burc} and
\cite[Cor.~2.2]{IKR}.

\begin{cor}
\label{cor:thm:Burchnall_general}
Under the general conditions of Section
\ref{sec:Burchnall_formula} we have
\begin{gather*}
G^{(\nu+n)}_{m} (G^{(\nu)}_{n+m})^{-1} P_{m+n}^{(\nu)}(x)=  \\
\sum_{k=0}^{n\wedge m}
\binom{n}{k}\binom{m}{k} k! \, P_{m-k}^{(\nu+n+k)}(x)   \bigl( G_{n-k}^{(\nu+k)}\bigr)^{-1}
P_{n-k}^{(\nu+k)}(x)  \left(\Phi^{(\nu)}(x) \cdots \Phi^{(\nu+k-1)}(x)\right)^\ast
\end{gather*}
\end{cor}

Next, we derive an integrated version of Burchnall's formula.

\begin{thm}
\label{thm:integrated_Burchnall}
Let $Q$ and $M$ be $C^\infty$ matrix valued functions
such that all of the integrals below converge. Then we have
\begin{multline*}
\int_a^b Q(x)W^{(\nu+n)}(x)M^{(n)}(x)dx= (-1)^n
\int_a^b \sum_{k=0}^n \binom{n}{k} Q^{(k)}(x)  \bigl( G_{n-k}^{(\nu+k)}\bigr)^{-1}
P_{n-k}^{(\nu+k)}(x) \\
 \times  \left(\Phi^{(\nu)}(x) \cdots \Phi^{(\nu+k-1)}(x)\right)^\ast W^{(\nu)}(x) M(x)\, dx.
\end{multline*}
\end{thm}

Note the special case $Q(x)=e^{-xt}I$, $M(x)=x^pI$
with $p<n$;
\begin{equation}\label{eq:thm:integrated_Burchnall}
\int_a^b \sum_{k=0}^n \binom{n}{k} (-t)^k  \bigl( G_{n-k}^{(\nu+k)}\bigr)^{-1}
P_{n-k}^{(\nu+k)}(x)
\left(\Phi^{(\nu)}(x) \cdots \Phi^{(\nu+k-1)}(x)\right)^\ast e^{-xt} W^{(\nu)}(x) x^p\, dx =0.
\end{equation}

\begin{proof}
Multiply \eqref{eq:Rodrigues_formal} by $W^{(\nu)}(x)M(x)$ from the right and we integrate with respect to $x$ on the interval $[a,b]$, we obtain as in the proof of Theorem \ref{thm:Burchnall_general}
$$\int_a^b \Bigl( QS^{(\nu+n-1)} \cdots S^{(\nu+1)}S^{(\nu)}\Bigr)(x) W^{(\nu)}(x)M(x)dx =
\int_a^b \frac{d^n}{dx^n}\Bigl( Q(x) W^{(\nu+n)}(x)\Bigr) M(x) dx$$
and the right hand side can
be evaluated by integration by parts as well as
by the proof of Theorem \ref{thm:Burchnall_general}.
So we get
\begin{equation*}
\begin{split}
 &(-1)^n \int_a^b Q(x)W^{(\nu+n)}(x)M^{(n)}(x)dx  \\
 = & \int_a^b \sum_{k=0}^n \binom{n}{k} Q^{(k)}(x)  \bigl( G_{n-k}^{(\nu+k)}\bigr)^{-1}  P_{n-k}^{(\nu+k)}(x) \, W^{(\nu+k)}(x) M(x)\, dx \\
= &  \int_a^b \sum_{k=0}^n \binom{n}{k} Q^{(k)}(x)  \bigl( G_{n-k}^{(\nu+k)}\bigr)^{-1}
P_{n-k}^{(\nu+k)}(x)  \left(\Phi^{(\nu)}(x) \cdots \Phi^{(\nu+k-1)}(x)\right)^\ast  W^{(\nu)}(x) M(x)\, dx
\qedhere
\end{split}
\end{equation*}
\end{proof}


\subsection{Application to matrix valued Gegenbauer-type polynomials}
\label{ssec:Gegenbauer}
The matrix valued Gegenbauer-type polynomials introduced in \cite{Koe:Rio:Rom} fit into this scheme.
For $\ell\in \frac12 \N$,  the weight is the $(2\ell+1)\times(2\ell+1)$-matrix which is given by
\begin{equation}
\label{eq:form_weight_Chebyshev}
W^{(\nu)}(x)=
L^{(\nu)}(x)T^{(\nu)}(x)L^{(\nu)}(x)^{t}, \qquad x\in(-1,1),
\end{equation}
where $L^{(\nu)}\colon [-1,1]\to M_{2\ell+1}(\C)$ is the unipotent lower triangular matrix valued polynomial
\[
\bigl(L^{(\nu)}(x)\bigr)_{m,k}=\begin{cases} 0 & \text{if } m<k \\
\displaystyle{\frac{m!}{k! (2\nu+2k)_{m-k}} C^{(\nu+k)}_{m-k}(x)} & \text{if } m\geq k.
\end{cases}
\]
Here, $C^{(\lambda)}_{m}(x)$ denote the classical Gegenbauer polynomials and $T^{(\nu)}\colon (-1,1)\to M_{2\ell+1}(\C)$ is the diagonal matrix valued function
\begin{equation*}
\begin{split}
\bigl(T^{(\nu)}(x)\bigr)_{k,k}\, =\, t^{(\nu)}_{k}\,  (1-x^2)^{k+\nu-1/2}, \quad
 t^{(\nu)}_{k}\, =\, \frac{k!\, (\nu)_k}{(\nu+1/2)_k}
\frac{(2\nu + 2\ell )_{k}\, (2\ell+\nu)}{(2\ell - k+1)_k\, ( 2\nu + k - 1)_{k}}.
\end{split}
\end{equation*}
Here we label the matrix entries by $m,k\in \{0,\cdots, 2\ell\}$.
The matrix valued Gegenbauer type orthogonal polynomials
satisfy \eqref{eq:Pearson_formal}, \eqref{eq:shift_formal},
\eqref{eq:rasing_on_Q}, \eqref{eq:Rodrigues_formal}, as follows
from Theorem~2.4, Corollary~2.2, Theorem~3.1 of \cite{Koe:Rio:Rom}. The explicit expression for the
second degree polynomial $\Phi^{(\nu)}$
and the first degree polynomial $\Psi^{(\nu)}$ are
given in \cite[(4.9), (4.10)]{Koe:Rio:Rom}.

If we use \eqref{eq:form_weight_Chebyshev} in the formula of Theorem \ref{thm:Burchnall_general}, and we use
\cite[Proposition 4.4]{Koe:Rio:Rom}, we obtain
\begin{multline}
\label{eq:Burchnall_Gegenbauer}
\Bigl(QS^{(\nu+n-1)} \cdots S^{(\nu+1)} S^{(\nu)}\Bigr)(x) = \sum_{k=0}^n \binom{n}{k} Q^{(k)}(x)  \bigl( G_{n-k}^{(\nu+k)}\bigr)^{-1}
P_{n-k}^{(\nu+k)}(x) \\
\times L^{(\nu+k)}(x) T^{(\nu+k)}(x) \left[ \left( M_1^{(\nu+k)}(x)\cdots M_1^{(\nu)}(x) \right)^t\right]^{-1} (T^{(\nu)}(x))^{-1} (L^{(\nu)}(x))^{-1},
\end{multline}
where $M^{(\nu)}_1(x)$ is a matrix polynomial
of second degree, see \cite[Prop.~4.4]{Koe:Rio:Rom}
for its explicit expression.

In particular, the expansion formula of
Corollary \ref{cor:thm:Burchnall_general} gives
an expansion for the matrix valued Gegenbauer
type polynomials which is an analogue of the
$\al=\be$ case of \cite[(3.11)]{IKR}.


\subsection{Application to matrix valued Hermite-type polynomials}
In Section \ref{sec:MV_Hermite}, we have shown that the Hermite-type matrix valued orthogonal polynomials satisfy \eqref{eq:Pearson_formal}, \eqref{eq:shift_formal},
\eqref{eq:rasing_on_Q}, \eqref{eq:Rodrigues_formal},
see in particular
Propositions \ref{prop:Pearson1},
\ref{prop:Pearson2}, \ref{prop:shift}.

In this case we can
use that in \eqref{eq:weight_factorization-alpha}
the matrix $L(x)$ does not depend on the parameter
$\nu$, and therefore the analog of \eqref{eq:Burchnall_Gegenbauer} can be simplified further;
$$
W^{(\al,\nu+k)}(x)(W^{(\al,\nu)}(x))^{-1}=
L^\al(x)\De^{(\nu+k)}(\De^{(\nu)})^{-1}(L^\al(x))^{-1}.
$$
Therefore we obtain
\begin{multline*}
\Bigl( \bigl( (QS^{(\nu+n-1)}) \cdots S^{(\nu+1)}\bigr)S^{(\nu)}\Bigr)(x)=\\
\sum_{k=0}^n \binom{n}{k} Q^{(k)}(x)
\bigl( G^{(\al,\nu)}_{n-k}\bigr)^{-1}
P_{n-k}^{(\al,\nu+k)}(x) \,
L^\al(x)\De^{(\nu+k)}(\De^{(\nu)})^{-1}(L^\al(x))^{-1}
\end{multline*}
as an alternative to Theorem \ref{thm:Burchnall_general}.
Similarly, the expansion formula in
Corollary \ref{cor:thm:Burchnall_general}
can be rewritten using the
$L^\al(x)\De^{(\nu+k)}(\De^{(\nu)})^{-1}(L^\al(x))^{-1}$
instead of
$\left(\Phi^{(\nu)}(x) \cdots \Phi^{(\nu+k-1)}(x)\right)^\ast$.


\section{An explicit solution to the
non-abelian Toda lattice}\label{sec:explsolNAToda}

The goal of this section is to give an
explicit solution of the non-abelian Toda
lattice equation
of Proposition \ref{prop:MV-Toda}
based on the matrix valued type
Hermite polynomials of Section \ref{sec:MV_Hermite}.
We determine the
monic orthogonal polynomials
with respect to the Toda modifications in two different ways.

\begin{prop}\label{prop:NATodasolHermite}
With the notation of Proposition \ref{prop:3termrecur} we have that
\begin{gather*}
C_n(t) = \exp(-\frac{t}{2}A^\al)C_n^{(\al,\nu)}\exp(\frac{t}{2}A^\al), \qquad
B_n(t) = \exp(-\frac{t}{2}A^\al)B_n^{(\al,\nu)}\exp(\frac{t}{2}A^\al) - \frac12 t
\end{gather*}
solve the non-abelian Toda lattice equation of
Proposition \ref{prop:MV-Toda} with $\La=I$.
\end{prop}

Since $A^\al$ does not commute with $B_n^{(\al,\nu)}$ nor with $C_n^{(\al,\nu)}$, this solution is
not trivial when compared to the corresponding solution of the Toda lattice equation, see e.g. \cite[Prop.~7.1(i)]{IKR}.

\begin{proof} Rewrite, using $t\in \R$,
\begin{multline*}
e^{-xt} L^\al(x) e^{-x^2} \De^{(\nu)}
\bigl( L^\al(x)\bigr)^\ast = \\
\exp(-\frac{t}{2}A^\al)e^{t^2/2}
L^\al(x+\frac12 t) e^{-(x+\frac12t)^2} \De^{(\nu)}
\bigl( L^\al(x+\frac12t)\bigr)^\ast
e^{t^2/2}\exp(-\frac{t}{2}A^\al)^\ast
\end{multline*}
using Lemma \ref{lem:Lm1JL}. The monic
orthogonal polynomials $P_n(x;t)$, cf. Proposition
\ref{prop:MV-Toda}, with respect to this matrix
measure are, up to a transformation, the matrix valued
Hermite type polynomials;
$$
P_n(x-\frac12t;t)\exp(-\frac{t}{2}A^\al)e^{t^2/2}
= \exp(-\frac{t}{2}A^\al)e^{t^2/2} P^{(\al,\nu)}_n(x).
$$
The constant (in $x$) matrix on the right hand side follows since the polynomials are monic.
From Proposition \ref{prop:3termrecur} we see that
\begin{gather*}
xP_n(x;t) = P_{n+1}(x;t)
+ \left( \exp(-\frac{t}{2}A^\al)B_n^{(\al,\nu)}\exp(\frac{t}{2}A^\al) - \frac12 t \right) P_n(x;t)\\
+ \exp(-\frac{t}{2}A^\al)C_n^{(\al,\nu)}\exp(\frac{t}{2}A^\al) P_{n-1}(x;t).
\end{gather*}
Now the result follows from Proposition \ref{prop:MV-Toda}.
\end{proof}

As a corollary to the proof, we have established in
this case
\begin{equation}\label{eq:MVTodaHermiteexplpols}
P_n(x;t)
= \exp(-\frac{t}{2}A^\al) P^{(\al,\nu)}_n(x+\frac12t)\exp(\frac{t}{2}A^\al).
\end{equation}
On the other hand, by \eqref{eq:thm:integrated_Burchnall}, in case
of the matrix valued Hermite type orthogonal polynomials, we see that
the polynomials, in the notation of Section \ref{sec:MV_Hermite},
\[
 \sum_{k=0}^n \binom{n}{k} (-t)^k  \bigl( G_{n-k}^{(\al,\nu+k)}\bigr)^{-1}
P_{n-k}^{(\al,\nu+k)}(x)
\left(\Phi^{(\al,\nu)}(x) \cdots \Phi^{(\al,\nu+k-1)}(x)\right)^\ast
\]
are orthogonal with respect the matrix weight $e^{-xt} W^{(\nu)}(x)$. Indeed, since the polynomial $\Phi^{(\al,\nu)}$
has degree $1$, this is indeed a polynomial of degree $n$.
The leading coefficient is
\begin{equation}\label{eq:lccombinationHermite}
M^{(\al,\nu)}_n(t) = \sum_{k=0}^n \binom{n}{k} t^k
\Bigl( \prod_{p=0}^{k-1} d^{(\nu+p)} \Bigr)
\bigl( G_{n-k}^{(\al,\nu+k)}\bigr)^{-1} (A^\al)^k
\end{equation}
which is a lower triangular matrix with
the constant (in $t$) matrix $\bigl( G_{n}^{(\al,\nu)}\bigr)^{-1}$ on the diagonal.
Hence $M^{(\al,\nu)}_n(t)$ is invertible.

\begin{prop}\label{prop:expansion-MVHermiteToda}
The matrix valued Hermite type polynomials $P^{(\al,\nu)}_n$ of
Section \ref{sec:MV_Hermite} satisfy
\begin{multline*}
\exp(-\frac{t}{2}A^\al) P^{(\al,\nu)}_n(x+\frac12t)\exp(\frac{t}{2}A^\al) = \\
\bigl(M^{(\al,\nu)}_n(t)\bigr)^{-1} \sum_{k=0}^n \binom{n}{k} (-t)^k  \bigl( G_{n-k}^{(\al,\nu+k)}\bigr)^{-1}
P_{n-k}^{(\al,\nu+k)}(x)
\left(\Phi^{(\al,\nu)}(x) \cdots \Phi^{(\al,\nu+k-1)}(x)\right)^\ast
\end{multline*}
\end{prop}

The identity of Proposition \ref{prop:expansion-MVHermiteToda} is the matrix analogue of \cite[(3.4)]{IKR} (which has a typo).
Now the scalar case \cite[(3.4)]{IKR} has an easy proof using the generating function \eqref{eq:genfunHermite},
but there does not seem to be a proof of Proposition \ref{prop:expansion-MVHermiteToda}
by generating functions.

Now that we have the Toda modified polynomials \eqref{eq:MVTodaHermiteexplpols}
explicitly, we can take the derivative with respect to $t$ and use Lemma \ref{lem:Todapolstimederivative}.
This gives, using the notation $X^{(\al,\nu)}_n(t)$ for the $t$-dependent one-but-leading coefficient
of $P_n(x;t)$ and Proposition \ref{prop:shift}(ii),
\begin{gather*}
e^{-\frac{t}{2}A^\al}\left( -\frac12 A^\al P^{(\al,\nu)}_n(x+\frac12t)
+\frac12 nP^{(\al,\nu+1)}_{n-1}(x+\frac12t) +P^{(\al,\nu)}_n(x+\frac12t) \frac12 A^\al \right) e^{\frac{t}{2}A^\al}
\\
= \frac{dX^{(\al,\nu)}_n}{dt}(t) e^{-\frac{t}{2}A^\al} P^{(\al,\nu)}_{n-1}(x+\frac12 t) e^{\frac{t}{2}A^\al}.
\end{gather*}
Simplifying gives
\begin{equation}
\begin{split}
nP^{(\al,\nu+1)}_{n-1}(x+\frac12t) + [P^{(\al,\nu)}_n(x+\frac12t), A^{\al}]=
2 e^{\frac{t}{2}A^\al} \frac{dX^{(\al,\nu)}_n}{dt}(t) e^{-\frac{t}{2}A^\al} P^{(\al,\nu)}_{n-1}(x+\frac12 t).
\end{split}
\end{equation}
Note that the $t$-derivative of the expression on the right hand side of
Proposition \ref{prop:expansion-MVHermiteToda} is more complicated.


\end{document}